\documentclass[a4paper,11pt]{amsart}
\usepackage{amssymb,amsmath}
\usepackage{amsopn}
\usepackage{mathrsfs}
\usepackage{color}
\usepackage{mathtools}
\usepackage{wasysym}
\usepackage{vmargin}
\usepackage{pdfsync}
\usepackage{vmargin}
\usepackage{color}
\usepackage{amscd}
\usepackage{verbatim}
\usepackage{bbm}
\usepackage{tikz}
\usepackage{hyperref}
\usetikzlibrary{patterns}
\usepackage{pgfplots}
\usepgfplotslibrary{fillbetween}

\newtheorem{theorem}{Theorem}[section]
\newtheorem{definition}[theorem]{Definition}
\newtheorem{lemma}[theorem]{Lemma}
\newtheorem{proposition}[theorem]{Proposition}

\numberwithin{equation}{section}

\newcommand{\rr}{\mathbb{R}}
\newcommand{\eps}{\varepsilon}
\newcommand{\nn}{\mathbb{N}}
\newcommand{\cc}{\mathbb{C}}
\def\un{{\mathrm{1~\hspace{-1.4ex}l}}}

\def\R{\mathbb R}

\def\val#1{\vert#1\vert}

\def\l2{L^2(\R^{n})}
\def\L2{L^2(\R^{2n})}

\def\eps{\varepsilon}

\def\mat22#1#2#3#4{\begin{pmatrix}#1&#2\\ #3&#4\end{pmatrix}}

\begin{document}
\title[Spectral inequalities for anisotropic Shubin operators]{Spectral inequalities for anisotropic Shubin operators}

\author{Jérémy \textsc{Martin}}

\address{\noindent \textsc{Jérémy Martin, Univ Rennes, CNRS, IRMAR - UMR 6625, F-35000 Rennes, France
}}
\email{jeremy.martin@univ-rennes1.fr}


\begin{abstract}
In this paper, new spectral inequalities for finite combinations of eigenfunctions of anisotropic Shubin operators are presented. We provide explicit spectral estimates for weak thick subsets and subsets with positive Lebesgue measure. The proofs are based on recent uncertainty principles holding in Gelfand-Shilov spaces and new Bernstein type estimates deduced from quantitative smoothing effects proved by Paul Alphonse. As an application, we deduce the null-controllability in any positive time from any control subset with positive Lebesgue measure for evolution equations associated to anisotropic Shubin operators, except for the harmonic oscillator.
\end{abstract}
\maketitle
\section{Introduction}
This work aims at providing quantitative spectral inequalities for anisotropic Shubin operators and null-controllability results for the evolution equations associated to these operators. Anisotropic Shubin operators are defined as the following self-adjoint elliptic operators
\begin{equation*}\label{Shubin}
H_{k,m}= (-\Delta_x)^m+|x|^{2k}, \quad x \in \rr^d
\end{equation*} 
where $|\cdot|$ denotes the Euclidean norm, equipped with the domain
\begin{equation*}
D(H_{k,m})=\big\{g \in L^2(\rr^d): \ H_{k,m}g \in L^2(\rr^d) \big\}, 
\end{equation*}
where $k,m \geq 1$ are positive integers. These operators are known to admit a non-decreasing sequence of positive eigenvalues $(\lambda_n^{k,m})_{n \in \nn}$ satisfying $\lambda_n^{k,m} \underset{n \to +\infty}{\longrightarrow} +\infty$ and an Hilbert basis of eigenfunctions $(\psi^{k,m}_n)_{n \in \nn}$ of $L^2(\rr^d)$, where $\psi^{k,m}_n$ is associated to the eigenvalue $\lambda^{k,m}_n>0$. For $\lambda>0$, $k, m \in \nn^*$, we consider $\mathcal E_{\lambda}^{k,m}$ which is the finite dimensional space defined as
\begin{equation}\label{finite_sum_eg}
\mathcal E_{\lambda}^{k,m} =\text{Span}_{\cc} \left\{\psi_n^{k,m}\right\}_{n \in \nn, \ \lambda_n \leq \lambda}.
\end{equation}

Our main purpose is to establish spectral inequalities of the following form 
\begin{equation}\label{generic_SI}
\forall \lambda >0, \forall f \in \mathcal{E}_{\lambda}^{k,m}, \quad \|f\|_{L^2(\rr^d)} \leq C_{k,m,\lambda}(\omega) \|f\|_{L^2(\omega)},
\end{equation}
where $\omega \subset \rr^d$ is a measurable subset and $C_{k,m, \lambda}(\omega) >0$ is a positive constant. The main result of this paper, namely Theorem~\ref{spec_ineq}, states such spectral inequalities under different geometric conditions on $\omega$. In particular, we provide quantitative constants $C_{k,m,\lambda}(\omega)$ with respect to the parameters $k, m, \lambda$. Mainly, our results deal with the thickness condition with respect to some density:

\medskip
\begin{definition}\label{weakly_thick}
Let $\omega \subset \rr^d$ be a measurable subset.\\
1) Let $0 < \gamma \leq 1$ and $\rho : \rr^d \longrightarrow (0,+\infty)$ be a positive function. The set $\omega$ is said to be $\gamma$-thick with respect to $\rho$ if and only if 
$$\forall x \in \rr^d, \quad |\omega \cap B(x, \rho(x))| \geq \gamma |B(x, \rho(x))|,$$
where $B(x, L)$ denotes the Euclidean ball of $\rr^d$ centered at $x$ with radius $L$ and $|A|$ denotes the Lebesgue measure of the measurable set $A$.\\
2) The subset $\omega$ is said to be $\delta$-weakly thick, for some $0 \leq \delta \leq 1$, if and only if there exist some positive constants $0< \gamma \leq 1$ and $R >0$ such that $\omega$ is $\gamma$-thick with respect to the density
\begin{equation*}
\rho(x)= R \langle x \rangle^{\delta}, \quad x \in \rr^d,
\end{equation*}
where $\langle x \rangle = \left(1+|x|^2\right)^{\frac12}$.
\end{definition}
\noindent Some examples and facts about the weak thickness condition are given in Section~\ref{thick_section}.\\

When $k=m=1$, the operator $H_{1,1}$ is the harmonic oscillator 
$$\mathcal{H}=-\Delta_x+|x|^2.$$
Thus, the Hilbert basis $(\psi^{1,1}_n)$ can be taken equal to the Hermite basis $(\Phi_{\alpha})_{\alpha \in \nn^d}$ given by 
$$\forall \alpha=(\alpha_1,...,\alpha_d) \in \nn^d, \forall x=(x_1,...,x_d) \in \rr^d, \quad \Phi_{\alpha}(x) = \prod_{i=1}^d H_i(x_i) e^{-\frac{x^2}{2}},$$
where for all $n \in \nn$, $H_n$ is the $n^{\text{th}}$-Hermite polynomial. In this case, spectral inequalities have been widely studied and recent results explicitly describe the growth of the constant $C_{1,1, \lambda}(\omega)$ appearing in \eqref{generic_SI} with respect to the energy level $\lambda$:
\begin{itemize}
\item In \cite[Theorem~2.1]{kkj}, Beauchard, Jaming and Pravda-Starov established that when $\omega$ is $0$-weakly thick (in that case, the set is said to be thick), the constant is of the form $C_{1,1, \lambda}(\omega)= K^{1+\sqrt \lambda}$ for some constant $K\geq1$ independent on $\lambda$. When the set $\omega$ satisfies the weaker geometric condition \begin{equation}\label{1weak_kkj}
\liminf_{R \to +\infty} \frac{|\omega \cap B(0,R)|}{|B(0,R)|} >0,
\end{equation}
the authors showed that the constant can be taken equal to $C_{1,1, \lambda}(\omega)= K^{1+\lambda}$ for some constant $K\geq1$ independent on $\lambda$. These authors have also obtained spectral inequalities for any non empty open subsets $\omega$ with a constant of the form $C_{1,1, \lambda}(\omega)=K^{1+\lambda |\log \lambda|}$.
\item In \cite[Theorem~2.1]{MP}, Pravda-Starov and the author showed that the constant can be taken equal to $C_{1,1, \lambda}(\omega)= K^{1+\lambda^{\delta}}$ as soon as $\omega$ is thick with respect to a $\frac12$-Lipschitz density $\rho$ satisfying 
$$\exists m, R>0, \forall x \in \rr^d, \quad 0<m\leq \rho(x) \leq R \langle x \rangle^{\delta},$$
with $0\leq \delta < 1$.
\item In \cite{dicke}, Dicke, Seelmann and Veselić generalised spectral inequalities given by \cite[Theorem~2.1]{MP}. Their work deal with a more general class of measurable subsets of $\rr^d$, called \textit{sensor sets}, which satisfy
\begin{equation}\label{thick_dicke}
\exists 0< \gamma \leq 1, \forall x \in \rr^d, \quad |\omega \cap B(x, \rho(x))| \geq \gamma^{1+|x|^{\eps’}} |B(x, \rho(x))|,
\end{equation}
where $\rho : \rr^d \longrightarrow (0,+\infty)$ is a continuous density satisfying
$$\exists R>0, \forall x \in \rr^d, \quad 0< \rho(x) \leq R \langle x \rangle ^{1-\eps},$$
with $0\leq \eps’ < \eps \leq 1$. Under this new geometric condition, they established in \cite[Theorem~2.7]{dicke} that the constant can be taken of the form $$C_{1,1, \lambda}(\omega)= \frac{1}{3} \Big(\frac{K}{\gamma}\Big)^{K^{1+\delta} d^{(11+3\delta)/2} (1+R^2) N^{1- \frac{\eps-\eps’}{2}}},$$
where $K>0$ is a positive universal constant.

\end{itemize}

In the case when $m=1$ and $k \geq 1$, the Shubin operator 
$$H_{k,1}= -\Delta_x +|x|^{2k}$$
is an anharmonic oscillator. In that case, Miller established in \cite{Miller0} spectral inequalities for finite sums of eigenfunctions associated to these operators when $\omega$ is an open cone:

\medskip
\begin{theorem}[{\cite[Corollary~1.6]{Miller0}}]\label{miller_thm}
Let $k \in \nn^*$. For any non empty open cone $$\omega = \big\{x \in \rr^d ; \ |x|>r_0,\; x/|x| \in \Omega_0\big\},$$ where $r_0 \geq 0$ and $\Omega_0$ is an open subset of the unit sphere, there is a constant $C>0$ such that
\begin{equation*}
\forall \lambda > 0, \forall g \in \mathcal{E}_{\lambda}^{k,1}, \quad \int_{\rr^d} |g(x)|^2 \mathrm dx \leq C e^{C \lambda^{\frac{1}{2}(1+\frac{1}{k})}} \int_{\omega} |g(x)|^2 \mathrm dx.
\end{equation*}
\end{theorem}
\medskip

\noindent Notice that the case $k=1$ is a particular case of the spectral inequalities given by \cite[Theorem~2.1]{kkj} under the geometric condition \eqref{1weak_kkj}, since the cone $\omega$ appearing in Theorem~\ref{miller_thm} satisfies this condition.
 
Through the Lebeau-Robbiano method, spectral inequalities have known a great interest in the control theory. This classical method, originally developed in \cite{lebeau_robbiano} in order to establish null-controllability results for the heat equation posed on a bounded domain $\Omega$ of $\rr^d$, requires to estimate finite sums of eigenfunctions. An adapted Lebeau-Robbiano method taken from \cite[Theorem~2.1]{egidi} is recalled in Appendix~\ref{control_appendix}, Theorem~\ref{Meta_thm_AdaptedLRmethod}. Following this strategy, Pravda-Starov and the author established in \cite[Theorem~2.5]{MP} that, given $\frac{1}{2}<s \leq 1$, the fractional harmonic heat equation
\begin{equation}\label{PDE_harmonic}
\left\lbrace \begin{array}{ll}
\partial_tf(t,x) + \mathcal{H}^s f(t,x)=u(t,x)\un_{\omega}(x), \quad &  x \in \mathbb{R}^d, \ t>0, \\
f|_{t=0}=f_0 \in L^2(\rr^d),                                       &  
\end{array}\right.
\end{equation}
is null-controllable in any time $T>0$ from any thick subset $\omega \subset \rr^d$ with respect to a $\frac{1}{2}$-Lipschitz density $\rho : \rr^d \longrightarrow (0,+\infty)$ verifying 
\begin{equation*}
\forall x \in \rr^d, \quad 0<m \leq \rho(x) \leq R{\left\langle x\right\rangle}^{\delta},
\end{equation*}
with $0\leq\delta <2s-1$. More generally, the authors of \cite[Theorem~3.2]{dicke} have shown that sensor sets defined by \eqref{thick_dicke} are efficient control subsets for fractional harmonic heat equations. In particular, these authors have constructed control subsets of finite Lebesgue measure which ensure the null-controllability of the harmonic heat equation. 

Such null-controllability questions for evolution equations associated to anharmonic oscillators
\begin{equation}\label{PDE_anharmonic}
\left\lbrace \begin{array}{ll}
\partial_tf(t,x) + (-\Delta_x +|x|^{2k}) f(t,x)=u(t,x)\un_{\omega}(x), \quad &  x \in \mathbb{R}^d, \ t>0, \\
f|_{t=0}=f_0 \in L^2(\rr^d),                                       &  
\end{array}\right.
\end{equation}
have been investigated by Miller in \cite{Miller0}. He obtained in \cite[Theorem~1.10]{Miller0} that for $k>1$, the evolution equation \eqref{PDE_anharmonic} is null-controllable in any time $T>0$ from any open cone $\omega$ defined in Theorem~\ref{miller_thm}. However, this result does not hold true for $k=1$, since \cite[Theorem~1.10]{Miller0} states that harmonic heat equation \eqref{PDE_harmonic} (with $s=1$) is not null-controllable at any time $T>0$ when $\omega$ is contained in a half space.

%
The main results contained in this work are given in Section~\ref{main_result} which is divided into two parts. The first part, Section~\ref{ineq_spec_results}, is devoted to present quantitative spectral inequalities for finite sums of eigenfunctions of anisotropic Shubin operators, given by Theorem~\ref{spec_ineq}. These spectral estimates generalize the ones given by \cite[Theorem~2.1]{MP}, \cite[Theorem~2.1]{kkj} and \cite[Corollary~1.6]{Miller0}. A second part, Section~\ref{control_result}, deals with the null-controllability of evolution equations associated to these operators. This work generalizes and improves known null-controllability results concerning these evolution equations. These new null-controllability results will be obtained as a consequence of Theorem~\ref{spec_ineq} and the Lebeau-Robbiano strategy. The proofs of these results are given in Section~\ref{proof_result} and \ref{proof_cont_shubin}. 

\section{Statements of the main results}\label{main_result}
This section is devoted to present the main results of this work. The first part states new spectral inequalities for finite combinations of eigenfunctions associated to anisotropic Shubin operators. The second one deals with the null-controllability of evolution equations associated to these operators.
\subsection{Spectral inequalities for anisotropic Shubin operators}\label{ineq_spec_results}
This section is devoted to state spectral inequalities for finite combinations of eigenfunctions associated to anisotropic Shubin operators.
One of the main results of this work is the following:
\medskip
\begin{theorem}\label{spec_ineq}
Let $k, m \geq 1$ be positive integers and $0\leq \delta \leq 1$. Let $\omega \subset \rr^d$ be a measurable subset. \\
\textit{(i)} If $\omega$ is $\delta$-weakly thick, then there exists a positive constant $K=K(\gamma, R, \delta, k, m)>0$ such that
\begin{equation*}
\forall \lambda >0, \forall f \in \mathcal{E}^{k,m}_{\lambda}, \quad \|f\|_{L^2(\rr^d)} \leq K e^{K \lambda^{\frac{1}{2}(\frac{\delta}{k}+\frac{1}{m})}} \|f\|_{L^2(\omega)},
\end{equation*}
where $0< \gamma \leq 1$ and $R>0$ are the constants appearing in Definition~\ref{weakly_thick}.
\medskip

\noindent
\textit{(ii)} If $\omega$ satisfies $|\omega|>0$, then there exists a positive constant $K=K(\omega, k,m)>0$ such that
\begin{equation*}
\forall \lambda >0, \forall f \in \mathcal{E}^{k,m}_{\lambda}, \quad \|f\|_{L^2(\rr^d)} \leq K e^{K \lambda^{\frac{1}{2}(\frac{1}{k}+\frac{1}{m})} |\log \lambda|} \|f\|_{L^2(\omega)},
\end{equation*}
where $\mathcal{E}^{k,m}_{\lambda}$ is defined in \eqref{finite_sum_eg}.
\end{theorem}
\medskip
Theorem~\ref{spec_ineq} establishes unique continuation properties by providing an explicit upper bound on the quantity
\begin{equation}\label{localisation_cost}
C_{\lambda}^{k,m}(\omega) = \sup_{f \in \mathcal{E}_{\lambda}^{k,m} \setminus \{0\}} \frac{\| f \|_{L^2(\rr^d)}}{\| f \|_{L^2(\omega)}}.
\end{equation}
This upper bound naturally depends on $k$, $m$ and the measurable subset $\omega$. In the case when $\omega$ is $\delta$-weakly thick for some $0 \leq \delta \leq 1$, the growth of this upper bound with respect to the energy level $\lambda$ is explicit and highlights a competition between $\delta$ and the parameter $k$. In the limit case $\delta=1$, this growth is ruled by the quantity $\frac{1}{2}(\frac{1}{k}+\frac{1}{m})$ which is strictly smaller than $1$ if and only if $(k,m) \neq (1,1)$. On the other hand, when $\omega$ is only a nonzero Lebesgue measure subset, this growth is comparable to the first case with $\delta=1$ but needs to be slightly modified by a factor $|\log \lambda|$. 

Notice that the results of \cite[Theorem~2.1]{kkj} and Theorem~\ref{miller_thm} are contained in the Theorem~\ref{spec_ineq}. As it is explained in the Section~\ref{thick_section}, the $1$-weak thick condition is equivalent to \eqref{1weak_kkj}, which is clearly satisfied by open cones. Moreover, this result covers more general geometric conditions, since the subset $\omega$ is not necessarly open contrary to Theorem~\ref{miller_thm} and deals with general anisotropic Shubin operators.  Let us mention that in the work \cite{Miller0}, the author wonders whether the inequalities given by Theorem~\ref{miller_thm} hold true for bounded open subsets with $\lambda^{\frac{1}{2}(1+\frac{1}{k})} \log \lambda$ instead of $\lambda^{\frac{1}{2}(1+\frac{1}{k})}$. Theorem~\ref{spec_ineq} assertion \textit{(ii)} positively answers this question.

The proof of Theorem~\ref{spec_ineq} is given in Section~\ref{proof_spec_ineq} and is based on Bernstein type estimates and on uncertainty principles holding in Gelfand-Shilov spaces established by the author in \cite{Martin}. The definition and some facts about Gelfand-Shilov spaces are recalled in Appendix~\ref{gelfand}. As an application of these spectral inequalities, we derive new null-controllability results for evolution equations associated to these operators in the next section.

\subsection{Applications to the null-controllability}\label{control_result}
\label{control_result}
In this section, we present null-controllability results for evolution equations associated to fractional anisotropic Shubin operators
\begin{equation*}\label{PDE_shubin0}
\left\lbrace \begin{array}{ll}
\partial_tf(t,x) + \big((-\Delta_x)^m+|x|^{2k}\big)^s f(t,x)=u(t,x)\un_{\omega}(x), \quad &  x \in \mathbb{R}^d, \ t>0, \\
f|_{t=0}=f_0 \in L^2(\rr^d),                                       &  
\end{array}\right.
\end{equation*}
where $s>0$ is a positive real parameter and $k, m \in \nn^*$ are positive integers.
Fractional Shubin operators are defined as follow
\begin{equation*}\label{frac_Shubin}
\forall g \in D(H_{k,m}^s), \quad H_{k,m}^s g= \sum_{n=0}^{+\infty} (\lambda_n^{k,m})^s \langle g, \psi^{k,m}_n \rangle_{L^2(\rr^d)} \psi^{k,m}_n,
\end{equation*}
equipped with the domain
\begin{equation*}
D(H_{k,m}^s)= \big\{g \in L^2(\rr^d): \ \sum_{n=0}^{+\infty} (\lambda_n^{k,m})^{2s} |\langle g, \psi^{k,m}_n \rangle_{L^2(\rr^d)}|^2 <+\infty \big\}.
\end{equation*}
These operators then generate a strongly continuous contraction semi-group on $L^2(\rr^d)$ given by 
\begin{equation*}\label{semigroup_Shubin}
\forall g \in L^2(\rr^d), \forall t \geq 0, \quad e^{-t H^s_{k,m}}g = \sum_{n=0}^{+\infty} e^{-t(\lambda_n^{k,m})^s} \langle g, \psi^{k,m}_n \rangle_{L^2(\rr^d)} \psi^{k,m}_n.
\end{equation*}

The notion of null-controllability is defined as follows:

\medskip

\begin{definition} [Null-controllability] Let $P$ be a closed operator on $L^2(\rr^d)$, which is the infinitesimal generator of a strongly continuous semigroup $(e^{-tP})_{t \geq 0}$ on $L^2(\rr^d)$, $T>0$ and $\omega$ be a measurable subset of $\mathbb{R}^d$. 
The evolution equation 
\begin{equation}\label{syst_general}
\left\lbrace \begin{array}{ll}
(\partial_t + P)f(t,x)=u(t,x)\un_{\omega}(x), \quad &  x \in \mathbb{R}^d,\ t>0, \\
f|_{t=0}=f_0 \in L^2(\rr^d),                                       &  
\end{array}\right.
\end{equation}
is said to be {\em null-controllable from the set $\omega$ in time} $T>0$ if, for any initial datum $f_0 \in L^{2}(\mathbb{R}^d)$, there exists a control function $u \in L^2((0,T)\times\mathbb{R}^d)$ supported in $(0,T)\times\omega$, such that the mild \emph{(}or semigroup\emph{)} solution of \eqref{syst_general} satisfies $f(T,\cdot)=0$.
\end{definition}

\medskip

 From the spectral inequalities stated by Theorem~\ref{spec_ineq}, we deduce the following null-controllability result:
 \medskip
\begin{theorem}\label{control_shubin_hd}
Let $k,m \geq 1$ be positive integers and $s >\frac{1}{2}(\frac{1}{k}+\frac{1}{m})$. Let $\omega \subset \rr^d$ be a measurable subset of positive Lebesgue measure $|\omega|>0$. The evolution system 
\begin{equation*}\label{PDE_shubin}
\left\lbrace \begin{array}{ll}
\partial_tf(t,x) + \big((-\Delta_x)^m+|x|^{2k}\big)^s f(t,x)=u(t,x)\un_{\omega}(x), \quad &  x \in \mathbb{R}^d, \ t>0, \\
f|_{t=0}=f_0 \in L^2(\rr^d),                                       &  
\end{array}\right.
\end{equation*}
is null-controllable from $\omega$ at any time $T>0$.
\end{theorem}
\medskip
Theorem~\ref{control_shubin_hd} ensures that evolution equations associated to fractional Shubin operators $H^s_{k,m}$ are always null-controllable as soon as $\omega$ is a nonzero Lebesgue measure subset and $s$ is strictly greater than a critical diffusion $s^*=\frac{1}{2}(\frac{1}{k}+\frac{1}{m})$. In particular, when the couple $(k,m) \neq (1,1)$, then $0\leq s^* <1$ and the evolution equation associated to $H_{k,m}$ is null-controllable at any time $T>0$ from $\omega \subset \rr^d$ as soon as $|\omega|>0$. 
This critical index $s^*$ turns out to be equal to $1$ if and only if $k=m=1$. In this case, the operator associated is the harmonic oscillator and Theorem~\ref{control_shubin_hd} provides that evolution equations associated to fractional harmonic oscillators $\mathcal{H}^s$, with $s>1$, is null-controllable at any time $T>0$ from any positive Lebesgue measure subset. This particular case was already proved by Alphonse in \cite[Theorem~2.9]{Alphonse}. 

The proof of Theorem~\ref{control_shubin_hd} is given in Section~\ref{proof_cont_shubin}. It consists in applying an adapted Lebeau-Robbiano method (Theorem~\ref{Meta_thm_AdaptedLRmethod}) together with the spectral inequalities given by Theorem~\ref{spec_ineq}.

Regarding the case when $s$ is smaller than the critical index $s^*$, spectral inequalities can be used to deduce the following null-controllability results:
\medskip
\begin{theorem}\label{control_shubin_sd}
Let $k,m \geq 1$ be positive integers, $\frac{1}{2m} < s \leq \frac{1}{2}(\frac{1}{k}+\frac{1}{m})$ and $0\leq \delta < \frac{k}{m}(2sm-1)$. If the measurable subset $\omega \subset \rr^d$ is $\delta$-weakly thick, then the evolution system 
\begin{equation*}\label{PDE_shubin}
\left\lbrace \begin{array}{ll}
\partial_tf(t,x) + \big((-\Delta_x)^m+|x|^{2k}\big)^s f(t,x)=u(t,x)\un_{\omega}(x), \quad &  x \in \mathbb{R}^d, \ t>0, \\
f|_{t=0}=f_0 \in L^2(\rr^d),                                       &  
\end{array}\right.
\end{equation*}
is null-controllable from $\omega$ at any time $T>0$.
\end{theorem}
\medskip
Theorem~\ref{control_shubin_sd} provides non-trivial efficient measurable subsets for the null-controllability of evolution equations associated to fractional Shubin operators. Actually, this Theorem was already proved by the author in \cite[Corollary~2.12]{Martin} and therefore the proof of Theorem~\ref{control_shubin_sd} is omitted in this work. The strategy developped by the author in \cite{Martin} consists in establishing quantitative uncertainty principles in Gelfand-Shilov spaces (see \cite[Theorem~2.2 and 2.3]{Martin}) in order to derive observability estimates for abstract evolution equations enjoying Gelfand-Shilov smoothing effects.
\section{The thickness property}\label{thick_section}
This section is devoted to present some facts about the weak thickness property. 
\subsection{Some properties of the weak thickness condition}
The first result establishes that the $\delta$-weak thickness property is non-decreasing with respect to the parameter $\delta$. 
\medskip
\begin{proposition}\label{increas}
Let $0 \leq \delta \leq \delta' \leq 1$ and $\omega \subset \rr^d$ be a measurable subset. If $\omega$ is $\delta$-weakly thick, then it is $\delta'$-weakly thick.
\end{proposition}
\medskip
\begin{proof}
Let us assume that $\omega$ is $\delta$-weakly thick. By definition, there exist $R>0$ and $0<\gamma \leq 1$ such that
\begin{equation}\label{thickdelta}
\forall x \in \rr^d, \quad \big|\omega \cap B\big(x, R \langle x \rangle^{\delta}\big)\big| \geq \gamma \big|B\big(x, R \langle x \rangle^{\delta}\big)\big|.
\end{equation}
Let $x \in \rr^d$. Let us first notice that
\begin{equation*}
\overline{B\big(x, 3 R \langle x \rangle^{\delta'}\big)} \subset \bigcup_{\substack{y \in \overline{B\big(x, 3 R \langle x \rangle^{\delta'}\big)}, \\ \langle y \rangle^{\delta} \leq 3 \langle x \rangle^{\delta'}}} B\big(y,  R \langle y \rangle^{\delta}\big)
\end{equation*}
Indeed, if $y \in \overline{B(x,3 R\langle x \rangle^{\delta'})}$ and $\langle y \rangle^{\delta} > 3 \langle x \rangle^{\delta'}$, then the continuous function defined for all $t \in [0,1]$ by $f(t)= \langle ty+(1-t)x \rangle^{\delta}$ satisfies $f(0)=\langle x\rangle^{\delta} \leq \langle x \rangle^{\delta'}$ and $f(1)=\langle y \rangle^{\delta} >3 \langle x \rangle^{\delta'}$. It follows that there exists $0 < t_0 < 1$ such that $\langle z\rangle^{\delta}=3 \langle x \rangle^{\delta'}$ with $z= t_0 y+ (1-t_0)x \in B\big(x,3 R \langle x\rangle^{\delta'}\big)$ and $y \in B\big(z,  R \langle z \rangle^{\delta}\big)$, as 
$$\|z-x\|=t_0\|x-y\| <3 R \langle x \rangle^{\delta'}, \qquad \|y-z\|=(1-t_0)\|x-y\| <3 R \langle x \rangle^{\delta'}=  R \langle z\rangle^{\delta}.$$
It follows that there exists a finite sequence $(x_{i_k})_{0 \leq k\leq N}$ of $\overline{B\big(x,3 R \langle x\rangle^{\delta'}\big)}$ such that 
\begin{equation}\label{recov30}
\overline{B\big(x,3 R \langle x\rangle^{\delta'}\big)} \subset \bigcup \limits_{k=0}^N B\big(x_{i_k},  R \langle x_{i_k} \rangle^{\delta}\big)\quad  \text{ and }\quad  \forall 0 \leq k \leq N, \quad \langle x_{i_k}\rangle^{\delta} \leq 3 \langle x\rangle^{\delta'}.
\end{equation} 
We can now use the following covering lemma~\cite{rudin} (Lemma~7.3):

\begin{lemma}[Vitali covering lemma] \label{rudin_recov}
Let $(y_i)_{0\leq i \leq N}$ be a finite sequence of $\rr^d$ and $(r_i)_{0 \leq i \leq N} \subset (0,+\infty)^{N+1}$. There exists a subset $S \subset \{0,...,N\}$ such that
\begin{itemize}
\item[$(i)$] The balls $(B(y_i, r_i))_{i \in S}$ are two by two disjoint 
\item[$(ii)$] $\bigcup \limits_{i=0}^N B(y_i, r_i) \subset \bigcup \limits_{i \in S} B(y_i, 3 r_i)$
\end{itemize}
\end{lemma}
It follows from Lemma~\ref{rudin_recov} and (\ref{recov30}) that there exists a subset $S \subset \{0,...,N\}$ such that the balls $\big(B\big(x_{i_k}, R \langle x_{i_k} \rangle^{\delta}\big)\big)_{k \in S}$ are two by two disjoint and satisfy
\begin{equation}\label{asdf10}
B\big(x,3 R \langle x\rangle^{\delta'}\big) \subset \bigcup \limits_{k \in S} B\big(x_{i_k},3 R \langle x_{i_k}\rangle^{\delta}\big).
\end{equation}
We also notice that 
\begin{equation}\label{asdf11}
\bigsqcup \limits_{k \in S} B\big(x_{i_k}, R \langle x_{i_k} \rangle^{\delta}\big) \subset B\big(x, 6 R \langle x\rangle^{\delta'}\big),
\end{equation}
since, if $y \in B\big(x_{i_k}, R \langle x_{i_k} \rangle^{\delta}\big)$ then
$$\|y-x\| \leq \|y-x_{i_k}\|+\|x_{i_k}-x\| < R \langle x_{i_k} \rangle^{\delta}+ 3 R \langle x \rangle^{\delta'} \leq 6 R \langle x \rangle^{\delta'}.$$
It follows from (\ref{thickdelta}), (\ref{asdf10}) and \eqref{asdf11} that
\begin{align*}
\big|\omega \cap B\big(x, 6 R \langle x \rangle^{\delta'}\big) \big| & \geq \sum_{k \in S} \big|\omega \cap B\big(x_{i_k}, R \langle x_{i_k} \rangle^{\delta}\big) \big| \\
& \geq \gamma \sum_{k \in S} \big|B\big(x_{i_k}, R \langle x_{i_k} \rangle^{\delta}\big) \big| = \frac{\gamma}{3^d} \sum_{k \in S} \big|B\big(x_{i_k},3 R \langle x_{i_k} \rangle^{\delta}\big) \big|\\
& \geq \frac{\gamma}{3^d} \Big| \bigcup_{k \in S} B\big(x_{i_k},3 R \langle x_{i_k} \rangle^{\delta}\big) \Big| \geq \frac{\gamma}{3^d} \big| B\big(x, 3 R \langle x \rangle^{\delta'}\big)\big| = \frac{\gamma}{6^d} \big| B\big(x, 6 R \langle x \rangle^{\delta'}\big)\big|.
\end{align*}
This ends the proof of Proposition~\ref{increas}.

\end{proof}

The following proposition provides a characterization of the $1$-weak thickness property and will be used in the proof of Theorem~\ref{spec_ineq}.
\medskip
\begin{proposition}\label{1weak}
Let $\omega \subset \rr^d$ be a measurable subset. The set $\omega$ is $1$-weakly thick if and only if $$\liminf \limits_{R \to +\infty} \frac{|\omega \cap B(0,R)|}{|B(0,R)|} >0.$$
\end{proposition}
\medskip
\begin{proof}
Let us first assume that $\omega$ is $1$-weakly thick. There exists $0< \gamma \leq 1$, $R_0>0$ such that
\begin{equation*}\label{1weak0}
\forall x \in \rr^d, \quad | \omega \cap B(x, R_0 \langle x \rangle)| \geq \gamma |B(x, R_0 \langle x \rangle)|.
\end{equation*}
Let $R \geq R_0$ and $x_R \in \rr^d$ such that $$|x_R|+R_0\langle x_R \rangle = R.$$
Let us notice that
$$|x_R| \leq R \quad \text{and} \quad \langle x_R \rangle \geq \frac{R}{R_0+1}.$$
Since $|x_R|+R_0\langle x_R \rangle = R$, the following inclusion holds
$$B(x_R, R_0 \langle x_R \rangle ) \subset B(0,R).$$
We deduce that
\begin{align*}
|\omega \cap B(0,R)| & \geq |\omega \cap B(x_R, R_0 \langle x_R \rangle)| \\ \nonumber
& \geq \gamma |B(x_R, R_0 \langle x_R \rangle)| = \gamma |B(0,1)| R_0^d \langle x_R \rangle^d \\ \nonumber
& \geq \frac{\gamma R_0^d}{(R_0+1)^d} R^d |B(0,1)| = \frac{\gamma R_0^d}{(R_0+1)^d} |B(0,R)|,
\end{align*}
and it readily follows that
$$\liminf \limits_{R \to +\infty} \frac{|\omega \cap B(0,R)|}{|B(0,R)|} \geq \frac{\gamma R_0^d}{R_0^d+1} >0.$$
Conversely, let us assume that 
$$\liminf \limits_{R \to +\infty} \frac{|\omega \cap B(0,R)|}{|B(0,R)|} >0.$$
It provides some constants $\delta >0$ and $R_0 >0$ such that
\begin{equation*}
\forall R \geq R_0, \quad |\omega \cap B(0,R)| \geq \delta |B(0,R)|.
\end{equation*}
Let $x \in \rr^d$. Since $|x| \leq \langle x \rangle$, we have the following inclusion
$$ B(0, R_0 \langle x \rangle) \subset B(x, (R_0+1) \langle x \rangle).$$
Since $R_0 \langle x \rangle \geq R_0$, we therefore deduce that
\begin{align*}
| \omega \cap B(x, (R_0+1) \langle x \rangle)| & \geq |\omega \cap B(0, R_0 \langle x \rangle)| \\ \nonumber
& \geq \delta |B(0, R_0 \langle x \rangle)| \\ \nonumber
& = \frac{\delta R_0^d}{(R_0+1)^d}  |B(x, (R_0+1) \langle x \rangle)|.
\end{align*}
This establishes that $\omega$ is $1$-weakly thick and this ends the proof of Proposition~\ref{1weak}.
\end{proof}
\subsection{Some examples of weak thick sets}
In this section, we give some examples of weak thick sets in the one and two-dimensional cases, which can be easily generalized to higher dimensions. Of course, any subset $\omega$ containing one of the following examples also satisfies a weak thickness condition. 
\medskip
  
In the one-dimensional case, let us consider the following measurable subset
\begin{equation*}\label{weak_thick_ex1d}
\omega_{\delta} = \bigcup_{n \in \nn} \Big[n^{\frac{1}{1-\delta}}, \frac{n^{\frac{1}{1-\delta}}+(n+1)^{\frac{1}{1-\delta}}}{2}\Big] \cup \Big[-\frac{n^{\frac{1}{1-\delta}}+(n+1)^{\frac{1}{1-\delta}}}{2}, -n^{\frac{1}{1-\delta}}\Big],
\end{equation*}
with $0 \leq \delta <1$. The set $\omega_{\delta}$ is $\delta$-weakly thick. However, when $\delta >0$, it fails to satisfy the $\delta’$-weak thickness condition, for any $0\leq \delta’ < \delta$, since the complementary of $\omega_{\delta}$ contains some intervals of the following form 
$$\Big]\frac{n^{\frac{1}{1-\delta}}+(n+1)^{\frac{1}{1-\delta}}}{2}, (n+1)^{\frac{1}{1-\delta}}\Big[, \quad n \in \nn,$$
of length
$$\frac{(n+1)^{\frac{1}{1-\delta}}-n^{\frac{1}{1-\delta}}}{2} \sim_{n \to +\infty} \frac{n^{\frac{\delta}{1-\delta}}}{2(1-\delta)}.$$
\begin{figure}[!h]
\begin{tikzpicture}[scale=1.3]
\draw[->] (-2.5,0) --(7.7,0);
\draw (7.7,0) node[below right] {$x$};
\draw [very thick, red] (-1,0) -- (-1.7,0);
\draw (-1,0.2) -- (-1,-0.2);
\draw [very thick, red] (-0.5,0) --(0.5,0);
\draw (0,0.2) -- (0,-0.2);
\draw (1,0.2) -- (1,-0.2);
\draw [very thick, red] (1,0) --(1.9,0);
\draw (2.8,0.2) -- (2.8,-0.2);
\draw [very thick, red] (2.8,0) --(4,0);
\draw (5.2,0.2) -- (5.2,-0.2);
\draw [very thick, red] (5.2,0) --(6.65,0);
\draw (-1,-0.2) node[below] {$-1$};
\draw (0,-0.2) node[below] {$0$};
\draw (1,-0.2) node[below] {$1$};
\draw (2.8,-0.2) node[below] {$2^{\frac{3}{2}}$};
\draw (5.2,-0.2) node[below] {$3^{\frac{3}{2}}$};
\end{tikzpicture}
\caption{The subset $\omega_{\frac{1}{3}}$. }
\end{figure}
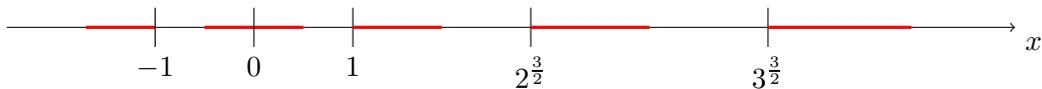

\begin{figure}[!h]
\begin{tikzpicture}[scale=1.3]
\draw[->] (-2.5,0) --(7.7,0);
\draw (7.7,0) node[below right] {$x$};
\draw [very thick, red] (-2,0) -- (-2.5,0);
\draw (-2,0.2) -- (-2,-0.2);
\draw [very thick, red] (-1,0) -- (-1.5,0);
\draw (-1,0.2) -- (-1,-0.2);
\draw [very thick, red] (-0.5,0) --(0.5,0);
\draw (0,0.2) -- (0,-0.2);
\draw (1,0.2) -- (1,-0.2);
\draw [very thick, red] (1,0) --(1.5,0);
\draw (2,0.2) -- (2,-0.2);
\draw [very thick, red] (2,0) --(2.5,0);
\draw (3,0.2) -- (3,-0.2);
\draw [very thick, red] (3,0) --(3.5,0);
\draw (4,0.2) -- (4,-0.2);
\draw [very thick, red] (4,0) --(4.5,0);
\draw (5,0.2) -- (5,-0.2);
\draw [very thick, red] (5,0) --(5.5,0);
\draw (6,0.2) -- (6,-0.2);
\draw [very thick, red] (6,0) --(6.5,0);
\draw (7,0.2) -- (7,-0.2);
\draw [very thick, red] (7,0) --(7.5,0);
\draw (-2,-0.2) node[below] {$-2$};
\draw (-1,-0.2) node[below] {$-1$};
\draw (0,-0.2) node[below] {$0$};
\draw (1,-0.2) node[below] {$1$};
\draw (2,-0.2) node[below] {$2$};
\draw (3,-0.2) node[below] {$3$};
\draw (4,-0.2) node[below] {$4$};
\draw (5,-0.2) node[below] {$5$};
\draw (6,-0.2) node[below] {$6$};
\draw (7,-0.2) node[below] {$7$};
\end{tikzpicture}
\caption{The subset $\omega_{0}$.}
\end{figure}
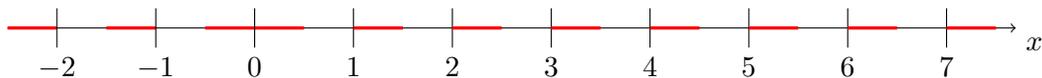

In the two-dimensional case, we define the following measurable subset of $\rr^2$
\begin{equation*}
\omega_{\delta, R} = \left\{ (x,y) \in \rr^2; \quad |y| > R \langle x\rangle^{\delta} \right\},
\end{equation*}
with $0 \leq \delta \leq1$ and $R>0$.
The subset $\omega_{\delta, R}$ is $\delta$-weakly thick. However, when $\delta >0$, its complement contains some balls of the form 
$$B(z, R’ \langle z \rangle^{\delta’}), \quad z \in \rr^2,$$
for any $R’>0$ and $0\leq \delta’ < \delta$. Then, this set can not be $\delta’$-weakly thick, for any $0\leq \delta’ <\delta$.
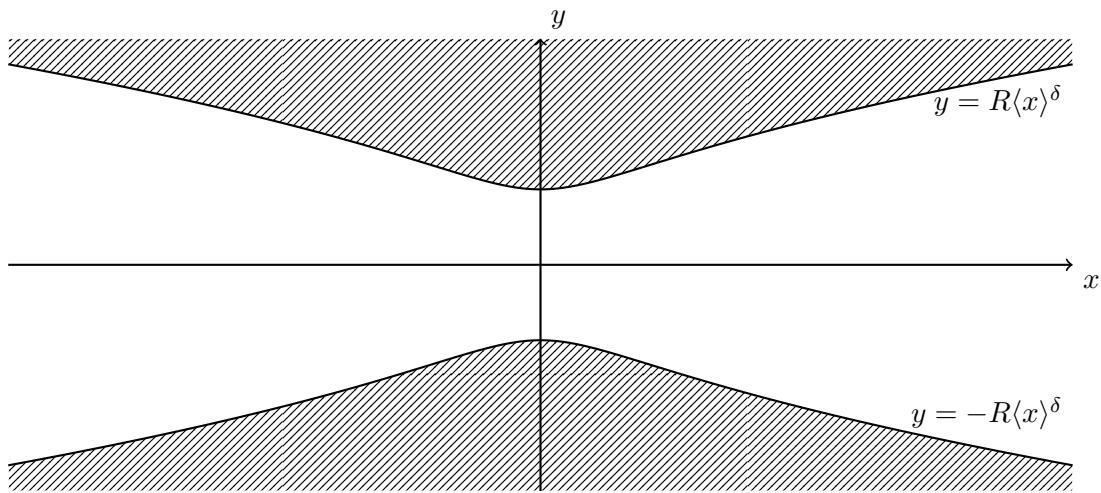
\begin{figure}[!h]
\begin{tikzpicture}[scale=1]


\filldraw[draw=white, fill=gray!40, pattern=north east lines]
(-7,{50^0.25})-- plot [domain=-7:7, samples=200] (\x, {(1+(\x)^2)^0.25})
-- (7, {50^0.25})-- (7, 3)-- (-7,3)--cycle;
\filldraw[draw=white, fill=gray!40, pattern=north east lines]
(-7,{-50^0.25})-- plot [domain=-7:7, samples=200] (\x, {-(1+(\x)^2)^0.25})
-- (7, {-50^0.25})-- (7, -3)-- (-7,-3)--cycle;
\draw [thick, domain=-7:7, samples=200] plot(\x, {(1+(\x)^2)^0.25});
\draw [thick, domain=-7:7, samples=200] plot(\x, {-(1+(\x)^2)^0.25});
\draw (7,{50^0.25-0.14}) node[below left] {$y=R\langle x \rangle^{\delta}$};
\draw (7,{-50^0.25+0.34}) node[above left] {$y=-R\langle x \rangle^{\delta}$};
\draw[thick, ->] (-7,0) --(7,0);
\draw (7,0) node[below right] {$x$};
\draw[thick, ->] (0,-3) --(0,3);
\draw (0,3) node[above right] {$y$};
\end{tikzpicture}
\caption{The subset $\omega_{\delta, R}$ with $0< \delta \leq 1$ and $R>0$.}
\end{figure}

We end this section by pointing out that contrary to the case $\delta <1$, the $1$-weak thickness property allows the subset $\omega$ to be contained in a half-space. Indeed, it can be readily check from Proposition~\ref{1weak} that any cone of the following form
\begin{equation*}
\mathcal C_{\theta}= \left\{(r \cos t, r \sin t) \in \rr^2; \ r>0, |t| \leq \frac{\pi}{2}-\theta \right\},
\end{equation*}
where $0 \leq \theta < \frac{\pi}{2}$, defines a $1$-weak thick set.


\begin{figure}[!h]
\begin{tikzpicture}[scale=1.4]
\draw (0,0)--({2*sqrt(3)},2);
\draw (0,0)--({2*sqrt(3)},-2);
\draw [->] (0,-2.3)->(0,2.3);
\draw [->] (-1.3,0)->(5,0);
\draw ({0.5*sqrt(3)},0.5) arc (30:90:1);
\draw ({0.5*sqrt(3)},-0.5) arc (-30:-90:1);
\draw (0.65,1.2) node[below]{$\theta$};
\draw (0.65,-1.2) node[above]{$\theta$};
\fill [pattern=north east lines] (0,0)--({2*sqrt(3)},2)--({2*sqrt(3)},-2);
\end{tikzpicture}  
\caption{The cone $\mathcal C_{\theta}$.}
\end{figure}
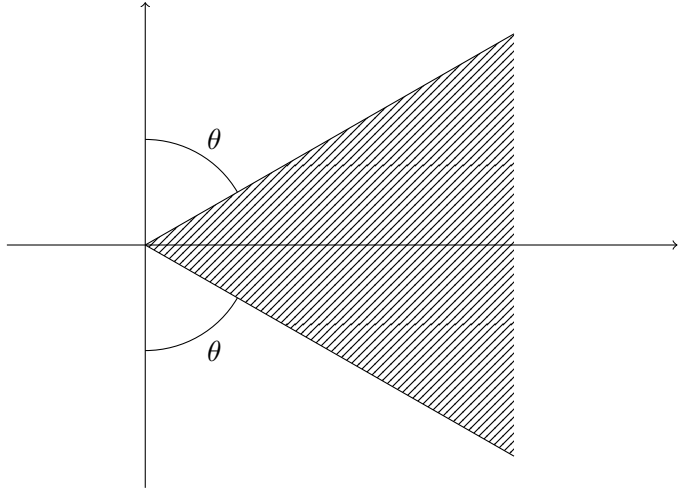
%

\section{Proof of Theorem~\ref{spec_ineq}}\label{proof_result}

This section aims at establishing the spectral inequalities given by Theorem~\ref{spec_ineq}. In Subsection~\ref{bernstein_estimate_proof}, we establish some Bernstein type estimates which are used in the proof of Theorem~\ref{spec_ineq}, given in Subsection~\ref{proof_spec_ineq}. In the following, for all multi-indexes $\alpha=(\alpha_1, \cdots, \alpha_d) \in \nn^d$, we use the notation
$|\alpha| =\alpha_1+\cdots+\alpha_d.$
\subsection{Smoothing properties and Bernstein type estimates}\label{bernstein_estimate_proof}
This section is devoted to the proof of the following Bernstein type estimates which play a key role in the proof of Theorem~\ref{spec_ineq}:
\medskip
\begin{proposition}\label{bernstein_estim}
Let $k,m \geq 1$ be positive integers and $0 < s \leq \frac{1}{2}(\frac{1}{k}+\frac{1}{m})$. There exist positive constants $C=C_{s, k,m,d}\geq 1$, $C'=C'_{s,k,m,d}\geq 1$, $C''=C''_{s,k,m,d}\geq 1$, $\eta=\eta_{s, k, m,d} >0$ and $\eta'=\eta'_{s,k,m,d}>0$ such that
\begin{multline}\label{norm2}
\forall \lambda >0, \forall g \in \mathcal{E}^{k,m}_{\lambda}, \forall p \in \nn, \forall \beta \in \nn^d, \\
\|\langle x\rangle^{p} \partial_x^{\beta} g \|_{L^2(\rr^d)} \leq {C}^{1+p+\val \beta} (p!)^{\frac{1}{2sk}} (\val \beta !)^{\frac{1}{2sm}} e^{\eta' \lambda^{s}} \|g\|_{L^2(\rr^d)},
\end{multline}
\begin{multline}\label{norm2bis}
\forall \lambda >0, \forall g \in \mathcal{E}^{k,m}_{\lambda}, \forall \beta \in \nn^d, \\
\|e^{\eta \langle x \rangle^{2sk}} \partial_x^{\beta} g \|_{L^2(\rr^d)} \leq {C'}^{1+\val \beta} (\val \beta !)^{\frac{1}{2sm}} e^{\eta' \lambda^{s}} \|g\|_{L^2(\rr^d)},
\end{multline}
\begin{align}\label{norminf}
\forall \lambda >0, \forall g \in \mathcal{E}^{k,m}_{\lambda}, \forall \beta \in \nn^d, \quad
\|\partial_x^{\beta} g \|_{L^{\infty}(\rr^d)} \leq {C''}^{1+\val \beta} (\val \beta !)^{\frac{1}{2sm}} e^{\eta' \lambda^{s}} \|g\|_{L^2(\rr^d)},
\end{align}
where $\mathcal{E}^{k,m}_{\lambda}$ is defined in \eqref{finite_sum_eg}.
\end{proposition}
\medskip
These estimates are derived from quantitative regularizing effects of semigroups generated by anisotropic Shubin operators.  Alphonse investigated the smoothing effects of these semigroups and established the following quantitative estimates in \cite[estimates (2.3)]{Alphonse}: 
\medskip
\begin{theorem}[Alphonse, \cite{Alphonse}]\label{smoothing_Shubin}
Let $k, m \geq 1$ be positive integers and $s >0$. Setting $s^*= \frac{1}{2}(\frac{1}{k}+\frac{1}{m}) \leq 1$, when $s \leq s^*$, there exist two positive constants $0 < t_s \leq 1$ and $C_s\geq 1$ such that 
\begin{multline*}
\forall 0< t \leq t_s, \forall g \in L^2(\rr^d), \forall \alpha, \beta \in \nn^d, \\
\|x^{\alpha} \partial_{x}^{\beta} (e^{-t\mathcal{H}^s_{k,m}}g) \|_{L^2(\rr^d)} \leq \frac{C^{1+\val \alpha+\val \beta}_s}{t^{\frac{\val \alpha}{2sk}+\frac{\val \beta}{2sm}+\frac{s^*}{s}d}} (\alpha!)^{\frac{1}{2sk}} (\beta!)^{\frac{1}{2sm}} \|g\|_{L^2(\rr^d)},
\end{multline*} 
whereas, when $s \geq s^*$, there exist two positive constants $0< t_s \leq 1$ and $C_s\geq 1$ such that 
\begin{multline*}
\forall 0< t \leq t_s, \forall g \in L^2(\rr^d), \forall \alpha, \beta \in \nn^d, \\
\|x^{\alpha} \partial_{x}^{\beta} (e^{-t\mathcal{H}^s_{k,m}}g) \|_{L^2(\rr^d)} \leq \frac{C^{1+\val \alpha+\val \beta}_s}{t^{\frac{m\val \alpha}{k+m}+\frac{k \val \beta}{k+m}+\frac{s^*}{s}d}} (\alpha!)^{\frac{m}{k+m}} (\beta!)^{\frac{k}{k+m}} \|g\|_{L^2(\rr^d)}.
\end{multline*} 
\end{theorem}
\medskip
We are now in position to prove Proposition~\ref{bernstein_estim}. Let $k, m \geq 1$ be positive integers, and $0< s \leq \frac{1}{2}(\frac{1}{k}+\frac{1}{m})$. According to Theorem~\ref{smoothing_Shubin}, there exist two positive constants $0< t_s \leq 1$ and $C_s \geq 1$ such that
\begin{multline}\label{smooth_eff}
\forall 0< t \leq t_s, \forall g \in L^2(\rr^d), \forall \alpha, \beta \in \nn^d, \\
\|x^{\alpha} \partial_{x}^{\beta} (e^{-t\mathcal{H}^s_{k,m}}g) \|_{L^2(\rr^d)} \leq \frac{C^{1+\val \alpha+\val \beta}_s}{t^{\frac{\val \alpha}{2sk}+\frac{\val \beta}{2sm}+\frac{s^*}{s}d}} (\alpha!)^{\frac{1}{2sk}} (\beta!)^{\frac{1}{2sm}} \|g\|_{L^2(\rr^d)}.
\end{multline} 
Let $\lambda >0$ and $g_0 \in \mathcal{E}^{k,m}_{\lambda}$. Let us define an auxiliary function 
\begin{equation*}
g= e^{t_s \mathcal{H}^s_{k,m}}g_0 := \sum_{\substack{n \in \nn, \\ \lambda_n^{k,m} \leq \lambda}} e^{t_s (\lambda_n^{k,m})^s} \langle g_0, \psi^{k,m}_n \rangle_{L^2(\rr^d)} \psi^{k,m}_n.
\end{equation*}
Notice that
$$e^{-t_s \mathcal{H}^s_{k,m}}g=g_0 \quad \text{and} \quad \| g \|_{L^2(\rr^d)} \leq e^{t_s \lambda^s} \| g_0 \|_{L^2(\rr^d)}.$$
By applying \eqref{smooth_eff} and the above remark, it follows that for all $\alpha, \beta \in \nn^d$,
\begin{align*}
\| x^{\alpha} \partial_x^{\beta} g_0 \|_{L^2(\rr^d)} & \leq \frac{C_s^{1+\val \alpha+\val \beta}}{t^{\frac{\val \alpha}{2sk}+\frac{\val \beta}{2sm}+\frac{s^*}{s}d}_s} (\alpha!)^{\frac{1}{2sk}} (\beta!)^{\frac{1}{2sm}} e^{t_s \lambda^s} \| g_0 \|_{L^2(\rr^d)} \\
& \leq C_{s,k,m,d}^{1+\val \alpha+\val \beta} (\alpha!)^{\frac{1}{2sk}} (\beta!)^{\frac{1}{2sm}} e^{t_s \lambda^s} \| g_0 \|_{L^2(\rr^d)} 
\end{align*}
with $$C_{s,k,m,d} =C_s \max\Big(t_s^{-\frac{s^*}{s}d},t_s^{-\frac{1}{2sk}}, t_s^{-\frac{1}{2sm}}, 1\Big)\geq 1.$$

\noindent We can now apply the technical Lemmas~\ref{croch} and \ref{interpolation}, stated and proved in Appendice~\ref{gelfand}.
First, it follows from Lemma~\ref{croch} that for all $p \in \nn$, $\beta \in \nn^d$,
\begin{equation}\label{firstestimate}
\|\left\langle x\right\rangle^p \partial_x^\beta g_0 \|_{L^2(\rr^d)} \leq (d+1)^{\frac{p}{2}} C_{s,k,m,d}^{1+p+\val \beta} (p!)^{\frac{1}{2sk}} (|\beta|!)^{\frac{1}{2sm}} e^{t_s \lambda^s} \| g_0 \|_{L^2(\rr^d)}.
\end{equation}
It proves the estimates \eqref{norm2}. Regarding the estimates \eqref{norm2bis}, we deduce from the fact that $0< s \leq 1$, \eqref{firstestimate} and Lemma~\ref{interpolation} 
that for all $p \in \nn$, $\beta \in \nn^d$,
\begin{multline}\label{interpolation1}
\|\langle x \rangle^{sk p} \partial_x^{\beta} g_0 \|_{L^2(\rr^d)} \\
\leq C_{s,k,m,d} (8^{\frac{1}{2sk}}e^{\frac{1}{2sk}} \sqrt{d+1} C_{s,k,m,d})^{kp+\val \beta} ((kp)!)^{\frac{1}{2k}} (\val \beta!)^{\frac{1}{2sm}} e^{t_s \lambda^s} \| g_0 \|_{L^2(\rr^d)}.
\end{multline}
By direct computations, we have for all $p \in \nn$,
\begin{equation*}
(kp)! \leq (kp)^{kp}= \big(k^{k}\big)^p \big(p^{p}\big)^{k} \leq \big((ke)^{k}\big)^p (p!)^{k},
\end{equation*}
since $$\forall p \in \nn, \quad p^p \leq p! e^p.$$
The above estimate, together with \eqref{interpolation1}, imply that there exists a new constant $C'=C'_{s, k, m, d} \geq 1$ such that for all $p \in \nn$, $\beta \in \nn^d$,
\begin{equation*}
\|\langle x \rangle^{sk p} \partial_x^{\beta} g_0 \|_{L^2(\rr^d)} \leq C'^{1+p+\val \beta} (p!)^{\frac{1}{2}} (\val \beta!)^{\frac{1}{2sm}} e^{t_s \lambda^s} \| g_0 \|_{L^2(\rr^d)}.
\end{equation*}
Let us define $\eta = \frac{1}{4C'^2}$. We deduce that for all $\beta \in \nn^d$,
\begin{align*}\label{norm2proof}
\| e^{\eta \langle x \rangle^{2sk}} \partial_x^{\beta} g_0 \|^2_{L^2(\rr^d)} & = \int_{\rr^d} e^{ 2\eta \langle x \rangle^{2sk}} |\partial_x^{\beta} g_0(x)|^2 dx \\ 
& = \sum_{p=0}^{+\infty} (2\eta)^p \frac{\|\langle x \rangle^{sk p} \partial_x^{\beta} g_0 \|^2_{L^2(\rr^d)}}{p!} \\
& \leq C'^{2(1+\val \beta)} (\val \beta!)^{\frac{1}{sm}} e^{2t_s \lambda^s} \| g_0 \|^2_{L^2(\rr^d)} \sum_{p=0}^{+\infty} \frac{1}{2^p} \\
& \leq 2 C'^{2(1+\val \beta)} (\val \beta!)^{\frac{1}{sm}} e^{2t_s \lambda^s} \| g_0 \|^2_{L^2(\rr^d)},
\end{align*} 
which proves the estimates \eqref{norm2bis}. It remains to establish \eqref{norminf}. First, by the multinomial formula, we have
\begin{equation*}
\forall j \in \nn, \quad (-\Delta_x)^j = \big( -\partial_{x_1}^2-...-\partial_{x_d}^2\big)^j = \sum_{\substack{\gamma \in \nn^d, \\ \val \gamma = j}} \frac{j!}{\gamma !} (-1)^j \partial_x^{2\gamma}.
\end{equation*}
We deduce from the above equality and \eqref{firstestimate} that for all $j \in \nn$,
\begin{align}\label{laplacian_estimate}
\| (-\Delta_x)^j g_0 \|_{L^2(\rr^d)} & \leq \sum_{\substack{\gamma \in \nn^d, \\ \val \gamma = j}} \frac{j!}{\gamma !} \| \partial_x^{2\gamma} g_0 \|_{L^2(\rr^d)} \\ \notag
& \leq \sum_{\substack{\gamma \in \nn^d, \\ \val \gamma = j}} \frac{j!}{\gamma !} C_{s,k,m,d}^{1+2\val \gamma} (\val{2\gamma} !)^{\frac{1}{2sm}} e^{t_s \lambda^s} \| g_0 \|_{L^2(\rr^d)} \\ \notag
& \leq C_{s,k,m,d}^{1+2j} d^j \big( (2j)!\big)^{\frac{1}{2sm}} e^{t_s \lambda^s} \| g_0 \|_{L^2(\rr^d)} \\ \notag
& \leq \left(2^{\frac{1}{2sm}}C_{s,k,m,d} \sqrt d\right)^{1+2j} (j!)^{\frac{1}{sm}} e^{t_s \lambda^s} \| g_0 \|_{L^2(\rr^d)},
\end{align}
since $$(2j)! =\binom{2j}{j} (j!)^2 \leq 2^{2j} (j!)^2.$$ Thanks to the Fourier inversion formula and Plancherel's formula, it follows that for all $\beta \in \nn^d$
\begin{align}\label{bern_norminf}
\| \partial_x^{\beta} g_0 \|_{L^{\infty}(\rr^d)} & \leq \frac{1}{(2\pi)^d} \int_{\mathbb R^d} \big\vert\xi^{\beta} \widehat g_0(\xi)\big\vert\ \mathrm d\xi \\[5pt] \nonumber
	& \leq \frac{1}{(2\pi)^d}\int_{\mathbb R^d} \frac{|\xi|^{\val \beta}+|\xi|^{\val \beta+d}}{1+|\xi|^{d}} |\widehat g_0(\xi)| \mathrm d\xi \\[5pt]
	& =\frac{1}{(2\pi)^d} \bigg(\int_{\mathbb{R}^d} \frac{1}{(1+|\xi|^{d})^2}\ \mathrm d\xi \bigg)^{\frac12} \big(\big\| |\xi|^{\val \beta}\widehat g_0 \big\|_{L^2(\rr^d)}+\big\||\xi|^{\val \beta+d} \widehat g_0 \big\|_{L^2(\mathbb{R}^d)}\big) \nonumber \\[5pt]
	& = \bigg(\int_{\mathbb{R}^d} \frac{1}{(1+|\xi|^{d})^2}\ \mathrm d\xi \bigg)^{\frac12} \left(\sqrt{\langle (-\Delta_x)^{\val \beta}g_0, g_0\rangle_{L^2(\rr^d)}}+\sqrt{\langle(-\Delta_x)^{\val \beta+d} g_0, g_0 \rangle_{L^2(\mathbb{R}^d)}}\right). \nonumber
\end{align} 
We therefore deduce from \eqref{bern_norminf}, together with \eqref{laplacian_estimate}, that for all $\beta \in \nn^d$,
\begin{equation*}
\|\partial_x^{\beta} g_0\|_{L^{\infty}(\rr^d)} \leq 2\bigg(\int_{\mathbb{R}^d} \frac{1}{(1+|\xi|^{d})^2}\ \mathrm d\xi \bigg)^{\frac12} {(2^{\frac{1}{2sm}}C_{s,k,m,d} \sqrt d)}^{1+\val \beta+d} ((\val \beta+d)!)^{\frac{1}{2sm}} e^{\frac{t_s}{2} \lambda^s} \|g_0\|_{L^2(\rr^d)}.
\end{equation*}
Since for all $\beta \in \nn^d, \quad (\val \beta+d)! \leq 2^{\val \beta+d} d! \val \beta!$, it follows that there exists a new constant $C'_{s, k,m, d} \geq 1$ such that
\begin{equation*}
\forall \beta \in \nn^d, \quad \|\partial_x^{\beta} g_0\|_{L^{\infty}(\rr^d)} \leq {C'}^{1+\val \beta}_{s,k,m,d} (\val \beta!)^{\frac{1}{2sm}} e^{\frac{t_s}{2} \lambda^s} \| g_0\|_{L^2(\rr^d)}.
\end{equation*}
This ends the proof of Proposition~\ref{bernstein_estim}.
\subsection{Proof of the spectral inequalities given by Theorem~\ref{spec_ineq}}\label{proof_spec_ineq}
This section is devoted to the proof of the spectral inequalities given by Theorem~\ref{spec_ineq}. Let us begin by presenting an elementary lemma, which ensures that the quantity $C_{\lambda}^{k,m}(\omega)$ defined in \eqref{localisation_cost} is finite, as soon as $|\omega|>0$ is positive:
\medskip
\begin{lemma}\label{equiv_norms}
Let $k, m \geq 1$ be positive integers and $\omega \subset \rr^d$ be a measurable subset of positive Lebesgue measure $|\omega|>0$. For all $\lambda>0$, there exists a positive constant $D_{k,m, \lambda}(\omega)>0$ such that
\begin{equation*}
\forall f \in \mathcal{E}^{k,m}_{\lambda}, \quad \|f\|_{L^2(\rr^d)} \leq D_{k, m , \lambda}(\omega) \|f\|_{L^2(\omega)},
\end{equation*}
where $\mathcal{E}^{k,m}_{\lambda}$ is defined in \eqref{finite_sum_eg}.
\end{lemma}
\medskip
\begin{proof}
Let $\lambda>0$. Since $\mathcal{E}^{k,m}_{\lambda}$ is a $\cc$-vector space of finite dimension, it is sufficient to show that $\| \cdot \|_{L^2(\omega)}$ is a norm on this space. Let $f \in \mathcal{E}^{k,m}_{\lambda}$ such that $\|f\|_{L^2(\omega)} =0$. By applying \eqref{norminf} in Proposition~\ref{bernstein_estim} with $s=\frac{1}{2}\big(\frac{1}{k}+\frac{1}{m}\big)$, we deduce that $f$ is analytic on $\rr^d$, since $\frac{1}{2sm}= \frac{k}{k+m}<1$. It follows that $f=0$ on $\omega$ and therefore, $f=0$, since $|\omega|>0$.
\end{proof}
\subsubsection{First case : $\omega$ is $\delta$-weakly thick with $0\leq \delta < 1$}
In this case, we use the following uncertainty principle with error term holding in Gelfand-Shilov spaces stated in \cite[Theorem~2.3]{Martin}:
\medskip
\begin{theorem}\label{specific_GS_uncertaintyprinciple}
Let $A \geq 1$, $0< \gamma \leq 1$, $0<\mu \leq 1$, $\nu >0$ with $\mu+\nu \geq 1$. Let $\rho : \rr^d \longrightarrow (0,+\infty)$ be a positive $\frac{1}{2}$-Lipschitz function such that there exist $m>0$, $R>0$ satisfying
\begin{equation*}
\forall x \in \rr^d, \quad 0<m \leq \rho(x) \leq R{\left\langle x\right\rangle}^{\frac{1-\mu}{\nu}}.
\end{equation*} 
Let $\omega \subset \rr^d$ be a measurable subset. If $\omega$ is $\gamma$-thick with respect to $\rho$, that is,
\begin{equation*}
\forall x \in \rr^d, \quad |\omega \cap B(x, \rho(x))| \geq \gamma |B(x, \rho(x))|,
\end{equation*} then for all $0<\eps \leq 1$, there exists a positive constant $K=K(d, \gamma, \rho, \mu, \nu) \geq 1$ such that for all $f \in \mathscr{S}(\rr^d)$,
 \begin{equation*}\label{up_schwartz}
 \| f \|^2_{L^2(\rr^d)} \leq e^{K(1-\log \eps+\log A)e^{KA^2}} \|f\|^2_{L^2(\omega)} + \eps \sup_{p \in \nn, \beta \in \nn^d} \left(\frac{\|\langle x\rangle^{p} \partial^{\beta}_{x} f\|_{L^2(\rr^d)}}{A^{p+\val \beta} (p!)^{\nu}(\val \beta!)^{\mu}}\right)^2.
 \end{equation*}
\end{theorem}
 \medskip
\noindent Let $\omega \subset \rr^d$ be a $\delta$-weakly thick subset. By definition, there exist $R>0$ and $0< \gamma \leq 1$ such that 
\begin{equation}\label{deltathick}
\forall x \in \rr^d, \quad |\omega \cap B(x, R \langle x \rangle^{\delta})| \geq \gamma |B(x, R \langle x \rangle^{\delta})|. 
\end{equation}
Let us check that $\omega$ satisfies the assumptions of Theorem~\ref{specific_GS_uncertaintyprinciple}. Notice that
\begin{equation*}
\nabla (R \langle x \rangle^{\delta})= R \delta \frac{x}{\langle x \rangle^{2-\delta}} \underset{|x| \to +\infty}{\longrightarrow} 0,
\end{equation*}
since $0\leq \delta < 1$, and therefore, there exists $L=L(R, \delta)>0$ such that 
\begin{equation}\label{estim_grad}
\forall |x| \geq L, \quad |\nabla (R \langle x \rangle^{\delta})|\leq \frac{1}{2}.
\end{equation}
 We define an auxiliary density by
\begin{equation*}
\rho(x) = \left\lbrace \begin{array}{ll}
R \langle x \rangle^{\delta} &  \text{if } |x| \geq L, \\
R (1+L^2)^{\frac{\delta}{2}}                       &  \text{if } |x| \leq L.
\end{array}\right.
\end{equation*}
Thanks to \eqref{estim_grad}, we can readily check that $\rho$ is a $\frac{1}{2}$-Lipschitz function and moreover, we have
\begin{equation}\label{aux1}
\forall x \in \rr^d, \quad 0< R \leq R \langle x \rangle^{\delta} \leq \rho(x) \leq R(1+L^2)^{\frac{\delta}{2}} \langle x \rangle^{\delta}.
\end{equation}
To apply Theorem~\ref{specific_GS_uncertaintyprinciple}, it remains to check that there exists $0<\tilde{\gamma}\leq 1$ such that
\begin{equation*}
\forall x \in \rr^d, \quad |\omega \cap B(x, \rho(x))| \geq \tilde{\gamma} |B(x, \rho(x))|.
\end{equation*}
Let $x \in \rr^d$. We deduce from \eqref{deltathick} and \eqref{aux1} that
\begin{align*}
|\omega \cap B(x, \rho(x))| \geq |\omega \cap B(x, R \langle x \rangle^{\delta})| & \geq \gamma |B(x, R \langle x \rangle^{\delta})| \\
& \geq \frac{\gamma}{(1+L^2)^{\frac{\delta d }{2}}} |B(x, \rho(x))|.
\end{align*}
We can therefore apply Theorem~\ref{specific_GS_uncertaintyprinciple} with the parameters $s=\frac{1}{2}\left(\frac{\delta}{k}+\frac{1}{m}\right)$, $0<\mu= \frac{1}{2sm}= \frac{k}{\delta m +k}\leq 1$ and $\nu= \frac{1}{2sk}= \frac{m}{\delta m+ k}>0$. Since $\delta= \frac{1-\mu}{\nu}$, we obtain that there exist some positive constants $K=K_{R, k, m, \delta, d}\geq 1$ and $D=D_{R, k, m, \delta, d} \geq 1$ such that for all $A\geq 1$, $0<\eps\leq 1$, $\lambda>0$, $f \in \mathcal{E}^{k,m}_{\lambda}$,
 \begin{equation}\label{up_apply}
 \| f \|^2_{L^2(\rr^d)} \leq e^{K(1-\log \eps+\log A)e^{KA^2}} \|f\|^2_{L^2(\omega)} + \eps \sup_{p \in \nn, \ \beta \in \nn^d} \frac{\| \langle x \rangle^p \partial_x^{\beta} f\|^2_{L^2(\rr^d)}}{A^{2(p+\val \beta)} (p!)^{\frac1{sk}} (\val \beta!)^{\frac1{sm}}}.
 \end{equation}
In order to estimate the error term in \eqref{up_apply}, we use the assertion \eqref{norm2} in Proposition~\eqref{bernstein_estim}. This provides positive constants $C=C_{s, k, m, d} \geq 1$ and $\eta’=\eta’_{s, k, m, d} >0$ such that for all $\lambda>0$, $f \in \mathcal{E}^{k,m}_{\lambda}$,
\begin{equation}\label{bernstein_estim290822}
\forall p \in \nn, \forall \beta \in \nn^d, \quad \| \langle x \rangle^p \partial^{\beta}_x f \|_{L^2(\rr^d)} \leq C^{1+p+|\beta|} (p!)^{\frac{1}{2sk}} (|\beta|!)^{\frac{1}{2sm}} e^{\eta’ \lambda^s} \| f \|_{L^2(\rr^d)}.
\end{equation}
It follows from \eqref{up_apply} and \eqref{bernstein_estim290822} with $A=C$ that for all $0<\eps\leq 1$, $\lambda>0$, $f \in \mathcal{E}^{k,m}_{\lambda}$,
 \begin{equation*}\label{up_apply1}
 \| f \|^2_{L^2(\rr^d)} \leq e^{K(1-\log \eps+\log C)e^{KC^2}} \|f\|^2_{L^2(\omega)} + \eps C^2 e^{2\eta’ \lambda^s} \|f\|^2_{L^2(\rr^d)} .
 \end{equation*}
By choosing $\eps= \frac{1}{2 C^2}e^{-2\eta' \lambda^s}$, we deduce that there exists a new constant $K’=K’_{R, k, m, \delta, d} \geq 1$ so that
 \begin{equation*}
\forall \lambda >0, \forall f \in \mathcal{E}^{k,m}_{\lambda}, \quad  \| f \|^2_{L^2(\rr^d)} \leq {K’}^{K’(1+\lambda^s)} \|f\|^2_{L^2(\omega)}.
 \end{equation*}
By recalling that $s=\frac{1}{2}\Big(\frac{\delta}{k}+\frac{1}{m}\Big)$, the result follows.

\subsubsection{Second case: $\omega$ is $1$-weakly thick}\label{1_weakly} 
In this case, since $\omega$ is $1$-weakly thick, we deduce from Proposition~\ref{1weak} that there exist $\gamma>0$ and $R_0>0$ such that
\begin{equation*}\label{weak_thick}
\forall R \geq R_0, \quad |\omega \cap B(0,R)| \geq \gamma |B(0,R)|.
\end{equation*}
As a first step, we establish that there exists a positive constant $c_{d, k,m}\geq 1$ such that 
\begin{equation*}
\forall \lambda \geq 1, \forall f \in \mathcal{E}^{k,m}_{\lambda}, \quad \| f \|_{L^2(\rr^d)}^2 \leq 2 \| f \|^2_{L^2\big(B(0, c_{d, k, m} \lambda ^{\frac{1}{2k}})\big)}.
\end{equation*}
Thanks to the assertions \eqref{norm2bis} and \eqref{norminf} in Proposition~\ref{bernstein_estim} with $s=\frac{1}{2}(\frac{1}{k}+\frac{1}{m})$, there exist some positive constants $\eta=\eta_{d,k,m}>0$, $\eta'=\eta’_{d,k,m}>0$ and $C=C_{d,k,m} \geq 1$ such that : $\forall \lambda >0, \forall f \in \mathcal{E}^{k,m}_{\lambda}, \forall \beta \in \nn^d,$
\begin{equation*}\label{bernstein_estim2}
\|e^{\eta \langle x \rangle^{1+\frac{k}{m}}} \partial_x^{\beta} f \|_{L^2(\rr^d)} \leq C^{1+\val \beta} (\val \beta!)^{\frac{k}{k+m}} e^{\eta' \lambda^{\frac{1}{2}(\frac{1}{k}+\frac{1}{m})}} \|f\|_{L^2(\rr^d)}
\end{equation*}
and
\begin{equation}\label{bernstein_estim3}
\|\partial_x^{\beta} f \|_{L^{\infty}(\rr^d)} \leq {C}^{1+\val \beta} (\val \beta!)^{\frac{k}{k+m}} e^{\eta' \lambda^{\frac{1}{2}(\frac{1}{k}+\frac{1}{m})}} \|f\|_{L^2(\rr^d)}.
\end{equation}
Let $a>0$. It follows that for all $\lambda >0$ and $f \in \mathcal{E}^{k,m}_{\lambda}$,
\begin{align}\label{space_decay}
\| f \|^2_{L^2(\rr^d)} & = \|f\|^2_{L^2(B(0,a))} + \| f \|^2_{L^2(\rr^d \setminus B(0,a))}\\ \notag
& \leq  \|f\|^2_{L^2(B(0,a))} + e^{-2\eta (1+a^2)^{\frac{1}{2}{(1+\frac{k}{m})}}} \| e^{\eta \langle x \rangle^{{1+\frac{k}{m}}}}f \|^2_{L^2(\rr^d)} \\ \notag
& \leq \|f\|^2_{L^2(B(0,a))} + e^{-2\eta a^{1+\frac{k}{m}}} C^2 e^{2\eta' \lambda^{\frac{1}{2}(\frac{1}{k}+\frac{1}{m})}} \|f \|^2_{L^2(\rr^d)}.
\end{align}
Let $c_{d, k, m}>0$ be a positive constant so that $$\forall \lambda \geq 1, \quad  e^{-2\eta a^{1+\frac{k}{m}}} C^2 e^{2\eta' \lambda^{\frac{1}{2}(\frac{1}{k}+\frac{1}{m})}} \leq \frac{1}{2},$$
with $$ a = c_{d, k, m} \lambda^{\frac{1}{2k}}.$$
It follows from \eqref{space_decay} that
\begin{equation}\label{space_decay1}
\forall \lambda \geq 1, \forall f \in \mathcal{E}^{k,m}_{\lambda}, \quad \| f \|^2_{L^2(\rr^d)} \leq 2 \|f\|^2_{L^2\big(B(0,c_{d, k, m} \lambda^{\frac{1}{2k}})\big)}.
\end{equation}
Let $\lambda\geq 1$ and $f \in \mathcal{E}^{k,m}_{\lambda} \setminus \{0\}$. Notice that, as a consequence of Lemma~\ref{equiv_norms}, we have
$$\forall r >0, \quad \|f\|_{L^2(B(0,r))} >0.$$
We can therefore set 
$$\forall x \in B\big(0,c_{d, k, m} \lambda^{\frac{1}{2k}}\big), \quad g(x)= \frac{f(x)}{\sqrt{2}C e^{\eta' \lambda^{\frac{1}{2}\big(\frac{1}{k}+\frac{1}{m}\big)}}\|f\|_{L^2\big(B(0,c_{d, k, m} \lambda^{\frac{1}{2k}})\big)}}.$$
By using the fact that
$$\| f \|_{L^2\big(B(0, c_{d,k,m}\lambda^{\frac{1}{2k}})\big)} \leq \sqrt{\left|B\big(0, c_{d, k, m} \lambda^{\frac{1}{2k}}\big)\right|} \| f \|_{L^{\infty}\big(B(0, c_{d,k,m}\lambda^{\frac{1}{2k}})\big)},$$
it follows that
\begin{equation}\label{borninf}
\|g\|_{L^{\infty}\big(B(0,c_{d, k, m} \lambda^{\frac{1}{2k}})\big)} \geq \frac{e^{-\eta' \lambda^{\frac{1}{2}\big(\frac{1}{k}+\frac{1}{m}\big)}}}{\sqrt{2|B(0,1)| c_{d,k,m}^d}C \lambda^{\frac{d}{4k}}} \geq e^{-\eta" \lambda^{\frac{1}{2}\big(\frac{1}{k}+\frac{1}{m}\big)}},
\end{equation}
for a suitable constant $\eta"=\eta"(d,k,m)\geq 1$ independent on $\lambda \geq 1$. 
Furthermore, we deduce from the estimates \eqref{bernstein_estim3} and \eqref{space_decay1} that
\begin{align}\label{bernstein_estim290822bis}
\forall \beta \in \nn^d, \quad \|\partial_x^{\beta} g \|_{L^{\infty}\big(B(0,c_{d, k, m} \lambda^{\frac{1}{2k}})\big)} & \leq \frac{\| \partial_x^{\beta} f \|_{L^{\infty}(\rr^d)}}{\sqrt{2}C e^{\eta'\lambda^{\frac{1}{2}\big(\frac{1}{k}+\frac{1}{m}\big)}} \|f\|_{L^2\big(B(0,c_{d, k, m} \lambda^{\frac{1}{2k}})\big)}} \\\nonumber
& \leq C^{\val \beta} (\val \beta!)^{\frac{k}{k+m}} \frac{\|f\|_{L^2(\rr^d)}}{\sqrt{2} \|f\|_{L^2\big(B(0,c_{d, k, m} \lambda^{\frac{1}{2k}})\big)}} \\\nonumber
& \leq C^{\val \beta} (\val \beta!)^{\frac{k}{k+m}}.
\end{align}
We can now use the following proposition established in \cite[Example~11]{Martin}: 
\medskip
\begin{proposition}\label{NSV_example}
Let $0< s < 1$, $A\geq 1$, $R>0$, $d \geq 1$, $0<t \leq 1$ and $0< \gamma \leq 1$ . Let $E \subset B(0,R)$ be a measurable subset of the Euclidean ball centered at $0$ with radius $R$ such that $|E| \geq \gamma |B(0,R)|$. There exists a constant $K=K(s, d) \geq 1$ such that for all $g \in \mathcal{C}^{\infty}(B(0,R))$ satisfying $$\forall \beta \in \nn^d, \quad \|\partial_x^{\beta} g \|_{L^{\infty}(B(0,R))} \leq A^{|\beta|} (|\beta|!)^s$$ and $\|g\|_{L^{\infty}(B(0,R))} \geq t$, the following estimate holds
\begin{equation*}
\| g \|_{L^2(B(0,R))} \leq C_{t, A, s, R, \gamma, d} \|g\|_{L^2(E)} ,
\end{equation*}
where
$$0<C_{t, A, s, R, \gamma, d} \leq \Big(\frac{K}{\gamma}\Big)^{K(1-\log t+ (AR)^{\frac{1}{1-s}})}.$$
\end{proposition}
\medskip
\noindent We deduce from Proposition~\ref{NSV_example}, \eqref{borninf} and \eqref{bernstein_estim290822bis} that there exists a new positive constant $\tilde{K}=\tilde{K}(d, k, m) \geq 1$ such that for all $\lambda \geq \lambda_0=\max\Big(\frac{R_0}{c_{d,k,m}},1\Big)^{2k}$,
\begin{equation*}
\| g \|_{L^2\big(B(0,c_{d, k, m}\lambda^{\frac{1}{2k}})\big)} \leq \Big(\frac{\tilde{K}}{\gamma}\Big)^{\tilde{K} \lambda^{\frac{1}{2}\big(\frac{1}{k}+\frac{1}{m}\big)}} \|g\|_{L^2\big(\omega \cap B(0,c_{d, k, m}\lambda^{\frac{1}{2k}})\big)}.
\end{equation*}
It follows from \eqref{space_decay1} that
\begin{align*}
\|f\|_{L^2(\rr^d)} & \leq \sqrt{2} \|f\|_{L^2\big(B(0,c_{d, k, m}\lambda^{\frac{1}{2k}})\big)} \\ \notag
& \leq \sqrt{2} \Big(\frac{\tilde{K}}{\gamma}\Big)^{\tilde{K}\lambda^{\frac{1}{2}\big(\frac{1}{k}+\frac{1}{m}\big)}} \|f\|_{L^2\big(\omega \cap B(0,c_{d, k, m}\lambda^{\frac{1}{2k}})\big)} \\ \notag
& \leq \sqrt{2} \Big(\frac{\tilde{K}}{\gamma}\Big)^{\tilde{K}\lambda^{\frac{1}{2}\big(\frac{1}{k}+\frac{1}{m}\big)}} \|f\|_{L^2(\omega)}.
\end{align*}
When $0< \lambda\leq \lambda_0= \max\Big(\frac{R_0}{c_{d,k,m}},1\Big)^{2k}$, Lemma~\ref{equiv_norms} implies that there exists $C_{k,m,\lambda_0}(\omega) >0$ such that $$\forall f \in \mathcal{E}^{k,m}_{\lambda}, \quad \|f\|_{L^2(\rr^d)} \leq C_{k,m, \lambda_0}(\omega) \|f\|_{L^2(\omega)},$$
since $$\forall 0< \lambda \leq \lambda_0, \quad  \mathcal{E}^{k,m}_{\lambda} \subset \mathcal{E}^{k,m}_{\lambda_0}.$$
By choosing $\tilde{\kappa}= \max (C_{k,m,\lambda_0}(\omega), \sqrt{2})$, it follows that 
$$\forall f \in \mathcal{E}^{k,m}_{\lambda}, \quad \|f\|_{L^2(\rr^d)} \leq \tilde{\kappa} \Big(\frac{\tilde{K}}{\gamma}\Big)^{\lambda^{\frac{1}{2}\big(\frac{1}{k}+\frac{1}{m}\big)}} \|f\|_{L^2(\omega)}.$$ 

\subsubsection{Third case : $|\omega|>0$}
Since $0<|\omega| \leq +\infty$, there exist $R_0>0, \sigma>0$ such that $$\forall R \geq R_0, \quad |\omega \cap B(0,R)| \geq \sigma>0.$$
With the same notations as in Section~\ref{1_weakly}, Proposition~\ref{NSV_example} implies once again that there exists a positive constant $\tilde{K}=\tilde{K}(d, k, m)\geq 1$ such that for all $\lambda \geq \lambda_0= \max\Big(\frac{R_0}{c_{d, k, m}},1\Big)^{2k}$, $f \in \mathcal{E}^{k,m}_{\lambda}$,
\begin{align*}
\| f \|_{L^2(\rr^d)} & \leq \sqrt{2} \bigg(\frac{\tilde{K}|B(0,1)|c^d_{d,k,m} \lambda^{\frac{d}{2k}}}{\sigma}\bigg)^{\tilde{K} \lambda^{\frac{1}{2}\big(\frac{1}{k}+\frac{1}{m}\big)}} \|f\|_{L^2\big(\omega \cap B(0,c_{d, k, m}\lambda^{\frac{1}{2k}})\big)} \\ \nonumber
& \leq \sqrt{2} \bigg(\frac{\tilde{K}|B(0,1)|c^d_{d,k,m} \lambda^{\frac{d}{2k}}}{\sigma}\bigg)^{\tilde{K} \lambda^{\frac{1}{2}\big(\frac{1}{k}+\frac{1}{m}\big)}} \|f\|_{L^2(\omega)}.
\end{align*}
It follows that there exists a positive constant $\tilde{K}'=\tilde{K}'(\omega,d,k,m) \geq 1$ such that for all $\lambda \geq \lambda_0$ and $f \in \mathcal{E}^{k,m}_{\lambda}$,
\begin{align*}
\|f \|_{L^2(\rr^d)}  & \leq \sqrt{2} \bigg(\frac{\tilde{K}|B(0,1)|c^d_{d,k,m}}{\sigma}\bigg)^{\tilde{K} \lambda^{\frac{1}{2}\big(\frac{1}{k}+\frac{1}{m}\big)}} \big(\lambda^{\frac{d}{2k}}\big)^{\tilde{K} \lambda^{\frac{1}{2}\big(\frac{1}{k}+\frac{1}{m}\big)}} \|f\|_{L^2(\omega)} \\ \nonumber
& \leq \big(\tilde{K}'\big)^{\tilde{K}' \lambda^{\frac{1}{2}\big(\frac{1}{k}+\frac{1}{m}\big)}} e^{\frac{\tilde{K} d}{2k} \lambda^{\frac{1}{2}\big(\frac{1}{k}+\frac{1}{m}\big)} \log \lambda} \|f\|_{L^2(\omega)}.
\end{align*}
It provides a new positive constant, still denoting by $\tilde{K}' \geq 1$, such that for all $\lambda \geq \max(\lambda_0, e)$ and $f \in \mathcal{E}^{k,m}_{\lambda}$,
\begin{equation*}
\|f \|_{L^2(\rr^d)} \leq  e^{\tilde{K}' \lambda^{\frac{1}{2}\big(\frac{1}{k}+\frac{1}{m}\big)} \log \lambda}  \|f\|_{L^2(\omega)}.
\end{equation*}
By using the very same lines as in the previous case, we deduce that for all $\lambda >0$ and $f \in \mathcal{E}^{k,m}_{\lambda}$,
\begin{equation*}
\| f \|_{L^2(\rr^d)}  \leq \tilde{\kappa} e^{\tilde{K}' \lambda^{\frac{1}{2}\big(\frac{1}{k}+\frac{1}{m}\big)}|\log \lambda|} \|f\|_{L^2(\omega)},
\end{equation*}
with $\tilde{\kappa}= \max (C_{k,m,\lambda_0}(\omega), 1)$.
This ends the proof of Theorem~\ref{spec_ineq}.
\section{Proof of the null-controllability result given by Theorem~\ref{control_shubin_hd}}\label{proof_cont_shubin}
This section is devoted to the proof of the null-controllability result given by Theorem~\ref{control_shubin_hd}.
Let $k, m \geq 1$, $s>0$ and $\omega \subset \rr^d$ be a measurable subset of positive Lebesgue measure $|\omega|>0$. Thanks to the Hilbert Uniqueness Method, the null-controllability of the evolution equation
\begin{equation*}
\left\lbrace \begin{array}{ll}
\partial_tf(t,x) + \big((-\Delta_x)^m+|x|^{2k}\big)^s f(t,x)=u(t,x)\un_{\omega}(x), \quad &  x \in \mathbb{R}^d, \ t>0, \\
f|_{t=0}=f_0 \in L^2(\rr^d),                                       &  
\end{array}\right.
\end{equation*}
is equivalent to the final state observability of the adjoint system
\begin{equation}\label{PDE_shubin_adjoint_proof}
\left\lbrace \begin{array}{ll}
\partial_tg(t,x) + \big((-\Delta_x)^m+|x|^{2k}\big)^s g(t,x)=0, \quad &  x \in \mathbb{R}^d, \ t>0, \\
g|_{t=0}=g_0 \in L^2(\rr^d).                                       &  
\end{array}\right.
\end{equation}
The definition of the notion of observability is recalled in Appendix~\ref{control_appendix}. To obtain the observability of the system \eqref{PDE_shubin_adjoint_proof}, we will use the Lebeau-Robbiano method, see Theorem~\ref{Meta_thm_AdaptedLRmethod} in Appendix~\ref{control_appendix}.
Let us assume that $s > \frac{1}{2}(\frac{1}{k}+\frac{1}{m})$. Thanks to Theorem~\ref{spec_ineq}, there exists a positive constant $K=K(\omega,d , k, m)\geq 1$ such that 
\begin{equation*}
\forall \lambda \geq 1, \forall f \in \mathcal{E}^{k,m}_{\lambda}, \quad \|f\|_{L^2(\rr^d)} \leq Ke^{K\lambda^{\frac{1}{2}(\frac{1}{k}+\frac{1}{m})} |\log \lambda|} \|f\|_{L^2(\omega)}.
\end{equation*}
Since $s>\frac{1}{2}(\frac{1}{k}+\frac{1}{m})$, there exist a positive parameter $\frac{1}{2}(\frac{1}{k}+\frac{1}{m})<s’<s$ and a new positive constant $K'=K'(\omega,d, s, s', k, m)>0$ such that
$$\forall \lambda \geq 1, \quad  Ke^{K\lambda^{\frac{1}{2}(\frac{1}{k}+\frac{1}{m})} \log(\lambda)} \leq K e^{K' \lambda^{s'}}.$$
We deduce that the spectral estimates required by Theorem~\ref{Meta_thm_AdaptedLRmethod} hold with the following choice of parameters : $c_1=K'$, $c_1'=K$, $a=s'$ and $\pi_{\lambda}$ defined as the orthogonal projection onto $\mathcal{E}^{k,m}_{\lambda}$. Moreover, we have the following dissipation estimates
\begin{equation*}
\forall \lambda >0, \forall t>0, \forall g \in L^2(\rr^d), \quad \|(1-\pi_{\lambda})(e^{-t \mathcal{H}_{k,m}^s}g)\|^2_{L^2(\rr^d)} \leq e^{-2t \lambda^s} \|g\|^2_{L^2(\rr^d)}.
\end{equation*}
Since $s>s'$, it follows from Theorem~\ref{Meta_thm_AdaptedLRmethod} that there exists a positive constant $C>1$ such that the following observability estimate holds
$$\forall T>0, \forall g \in L^2(\rr^d), \quad \|e^{-T \mathcal{H}_{k,m}^s}g\|^2_{L^2(\rr^d)} \leq C \exp\Big( \frac{C}{T^{\frac{s'}{s-s'}}}\Big) \int_0^T \|e^{-t\mathcal{H}_{k,m}^s}g\|^2_{L^2(\omega)} dt.$$
This ends the proof of Theorem~\ref{control_shubin_hd}.

\section{Appendix}

\subsection{Gelfand-Shilov regularity}\label{gelfand}
This appendix is devoted to recall basic properties about Gelfand-Shilov spaces and to present technical lemmas used in the proof of Proposition~\ref{bernstein_estim}.
We refer the reader to the works~\cite{gelfand_shilov,rodino1,rodino,toft} and the references herein for extensive expositions of the Gelfand-Shilov regularity theory.
The Gelfand-Shilov spaces $S_{\nu}^{\mu}(\rr^d)$, with $\mu,\nu>0$, $\mu+\nu\geq 1$, are defined as the spaces of smooth functions $f \in C^{\infty}(\rr^d)$ satisfying the estimates
$$\exists A,C>0, \quad |\partial_x^{\alpha}f(x)| \leq C A^{\val \alpha}(\alpha !)^{\mu}e^{-\frac{1}{A}|x|^{1/\nu}}, \quad x \in \rr^d, \ \alpha \in \mathbb{N}^d,$$
or, equivalently
$$\exists A,C>0, \quad \sup_{x \in \rr^d}|x^{\beta}\partial_x^{\alpha}f(x)| \leq C A^{\val \alpha+\val \beta}(\alpha !)^{\mu}(\beta !)^{\nu}, \quad \alpha, \beta \in \mathbb{N}^d,$$
with $\alpha!=(\alpha_1!)...(\alpha_d!)$ if $\alpha=(\alpha_1,...,\alpha_d) \in \nn^d$.
These Gelfand-Shilov spaces  $S_{\nu}^{\mu}(\rr^d)$ may also be characterized as the spaces of Schwartz functions $f \in \mathscr{S}(\rr^d)$ satisfying the estimates
$$\exists C>0, \eps>0, \quad |f(x)| \leq C e^{-\eps|x|^{1/\nu}}, \quad x \in \rr^d; \qquad |\widehat{f}(\xi)| \leq C e^{-\eps|\xi|^{1/\mu}}, \quad \xi \in \rr^d.$$

In particular, we notice that Hermite functions belong to the symmetric Gelfand-Shilov space  $S_{1/2}^{1/2}(\rr^d)$. More generally, the symmetric Gelfand-Shilov spaces $S_{\mu}^{\mu}(\rr^d)$, with $\mu \geq 1/2$, can be nicely characterized through the decomposition into the Hermite basis $(\Phi_{\alpha})_{\alpha \in \mathbb{N}^d}$, see e.g. \cite[Proposition~1.2]{toft},
\begin{multline*}
f \in S_{\mu}^{\mu}(\rr^d) \Leftrightarrow f \in L^2(\rr^d), \ \exists t_0>0, \ \big\|\big(\langle f,\Phi_{\alpha}\rangle_{L^2}\exp({t_0\val \alpha^{\frac{1}{2\mu}})}\big)_{\alpha \in \mathbb{N}^d}\big\|_{l^2(\mathbb{N}^d)}<+\infty\\
\Leftrightarrow f \in L^2(\rr^d), \ \exists t_0>0, \ \|e^{t_0\mathcal{H}^{\frac{1}{2\mu}}}f\|_{L^2(\rr^d)}<+\infty,
\end{multline*}
where $\mathcal{H}=-\Delta_x+|x|^2$ stands for the harmonic oscillator. More generally, when the ratio $\frac{\mu}{\nu} \in \mathbb{Q}$ is a rational number, the Gelfand-Shilov space $S^{\mu}_{\nu}(\rr^d)$ can also be nicely characterized through the decomposition into the orthonormal basis of eigenfunctions of a class of anisotropic Shubin operators, whose basic model is the operator
$$\mathcal{H}_{k,m}=(-\Delta_x)^m+|x|^{2k}, \quad x \in \rr^d,$$
with $k, m \geq 1$ two positive integers. Let $(\psi^{k,m}_n)_{n \geq 1}$ be an orthonormal basis of $L^2(\rr^d)$ composed of eigenfunctions of the above operator $\mathcal{H}_{k,m}$. Given a real number $t \geq 1$, the Gelfand-Shilov space $S^{\mu}_{\nu}(\rr^d)$, with
$$\mu=\frac{kt}{k+m} \quad \text{and} \quad \nu=\frac{mt}{k+m},$$
are characterized in the following way, thanks to the result \cite[Theorem~1.4]{cappiello} by Cappiello, Gramchev, Pilipović and Rodino, 
$$f \in S^{\frac{kt}{k+m}}_{\frac{mt}{k+m}}(\rr^d) \iff \exists \eps >0, \quad \sum_{n=1}^{+\infty} |\langle f, \psi^{k,m}_n \rangle_{L^2(\rr^d)}|^2 e^{\eps \lambda_n^{\frac{k+m}{2kmt}}} <+\infty,$$
where $\lambda_n>0$ is the eigenvalue associated to the eigenfunction $\psi^{k,m}_n \in L^2(\rr^d)$ for all $n \geq 1$. Such a characterization in the case when $\mu/\nu \notin \mathbb Q$ has not been found yet.
We end this section by proving two technical lemmas:

\begin{lemma}\label{croch}
Let $\mu, \nu >0$ such that $\mu+\nu \geq 1$, $C>0$ and $A \geq 1$. If $f \in S_{\nu}^{\mu}(\rr^d)$ satisfies
\begin{equation}\label{gs_estim}
\forall \alpha \in \nn^d, \forall \beta \in \nn^d, \quad \| x^{\alpha} \partial_x^{\beta} f \|_{L^2(\rr^d)} \leq C A^{\val \alpha+\val \beta} (\alpha!)^{\nu} (\beta!)^{\mu},
\end{equation}
then, it satisfies 
\begin{equation*}
\forall p \in \nn, \forall \beta \in \nn^d, \quad \|\langle x \rangle^{p} \partial_x^{\beta} f \|_{L^2(\rr^d)} \leq C (d+1)^{\frac{p}{2}}A^{p+\val \beta} (p!)^{\nu} (\val \beta!)^{\mu}.
\end{equation*}
\end{lemma}
\begin{proof}
Let $f \in S_{\nu}^{\mu}(\rr^d)$ satisfying the estimates \eqref{gs_estim}. By using Newton formula, we obtain that for all $p \in \nn$, $\beta \in \nn^d$,
\begin{multline}\label{croch_estim}
\|\langle x \rangle^p \partial_x^{\beta} f \|^2_{L^2(\rr^d)} = \int_{\rr^d} \Big(1+\sum_{i=1}^d x_i^2 \Big)^p |\partial_x^{\beta}f(x)|^2 dx \\
= \int_{\rr^d} \sum_{\substack{\gamma \in \nn^{d+1}, \\ \val \gamma=p}} \frac{p!}{\gamma!} x^{2\tilde{\gamma}} |\partial_x^{\beta}f(x)|^2 dx = \sum_{\substack{\gamma \in \nn^{d+1}, \\ \val \gamma=p}} \frac{p!}{\gamma!} \|x^{\tilde{\gamma}} \partial_x^{\beta} f \|^2_{L^2(\rr^d)},
\end{multline}
where we denote $\tilde{\gamma}=(\gamma_1,...,\gamma_d) \in \nn^d$ if $\gamma=(\gamma_1,...,\gamma_{d+1}) \in \nn^{d+1}$. Since for all $\alpha \in \nn^d$, $\alpha! \leq (\val \alpha)!$, it follows from \eqref{gs_estim} and \eqref{croch_estim} that
\begin{align*}
\|\langle x \rangle^p \partial_x^{\beta} f \|^2_{L^2(\rr^d)} & \leq C^2\sum_{\substack{\gamma \in \nn^{d+1}, \\ \val \gamma=p}} \frac{p!}{\gamma!} A^{2(\val{\tilde{\gamma}}+\val \beta)} (\val{\tilde{\gamma}}!)^{2\nu} (\val \beta!)^{2\mu} \\
& \leq C^2 (d+1)^p A^{2(p+\val \beta)} (p!)^{2\nu} (\val \beta!)^{2\mu},
\end{align*}
since $$ \sum_{\substack{\gamma \in \nn^{d+1}, \\ \val \gamma=p}} \frac{p!}{\gamma!}  = (d+1)^p.$$
\end{proof}
\begin{lemma}\label{interpolation}
Let $\mu, \nu >0$ such that $\mu+\nu \geq 1$, $0 \leq \delta \leq 1$, $C>0$ and $A \geq 1$. If $f \in S_{\nu}^{\mu}(\rr^d)$ satisfies
\begin{equation}\label{int}
\forall p \in \nn, \forall \beta \in \nn^d, \quad \|\langle x \rangle^p \partial_x^{\beta} f \|_{L^2(\rr^d)} \leq C A^{p+\val \beta} (p!)^{\nu} (\val \beta!)^{\mu},
\end{equation}
then, it satisfies 
\begin{equation*}
\forall p \in \nn, \forall \beta \in \nn^d, \quad \|\langle x \rangle^{\delta p} \partial_x^{\beta} f \|_{L^2(\rr^d)} \leq C(8^{\nu} e^{\nu}A)^{p+\val \beta} (p!)^{\delta \nu} (\val \beta!)^{\mu}.
\end{equation*}
\end{lemma}
\begin{proof}
Let $f \in S_{\nu}^{\mu}(\rr^d)$ satisfying the estimates \eqref{int}. 
It follows from H\"older inequality that for all $r \in (0,+\infty) \setminus \nn$ and $\beta \in \nn^d$,
\begin{multline}\label{holder0}
\|\langle x \rangle^r \partial_x^{\beta} f \|^2_{L^2(\rr^d)} = \int_{\rr^d} \big(\langle x \rangle^{2 \lfloor r \rfloor} |\partial^{\beta}_x f(x)|^2\big)^{\lfloor r \rfloor +1-r} \big(\langle x \rangle^{2(\lfloor r \rfloor+1)} |\partial^{\beta}_x f(x)|^2\big)^{r-\lfloor r \rfloor} dx \\
\leq \|\langle x \rangle^{\lfloor r \rfloor} \partial_x^{\beta} f \|^{2(\lfloor r \rfloor +1- r)}_{L^2(\rr^d)} \|\langle x \rangle^{\lfloor r \rfloor+1} \partial_x^{\beta} f \|^{2(r-\lfloor r \rfloor)}_{L^2(\rr^d)},
\end{multline}
where $\lfloor \cdot \rfloor$ denotes the floor function.
Since the above inequality clearly holds for $r \in \nn$, we deduce from \eqref{int} and \eqref{holder0} that for all $r \geq 0$ and $\beta \in \nn^d$,
\begin{align}\label{GS_1}
\|\langle x \rangle^r \partial_x^{\beta} f \|_{L^2(\rr^d)} & \leq C A^{r+\val \beta} (\lfloor r \rfloor!)^{(\lfloor r \rfloor +1-r) \nu} \big((\lfloor r \rfloor+1)!\big)^{(r-\lfloor r \rfloor) \nu} (\val \beta!)^{\mu} \\ \nonumber
& \leq C A^{r+\val \beta} \big((\lfloor r \rfloor+1)!\big)^{\nu} (\val \beta!)^{\mu} \\ \nonumber
& \leq  C A^{r+\val \beta} (\lfloor r \rfloor+1)^{(\lfloor r \rfloor+1)\nu} (\val \beta!)^{\mu} \\ \nonumber
& \leq C A^{r+\val \beta} (r+1)^{(r+1)\nu} (\val \beta!)^{\mu}.
\end{align}
It follows from \eqref{GS_1} that for all $p \in \nn^*$, $\beta \in \nn^d$,
\begin{align}\label{puiss_frac}
\|\langle x \rangle^{\delta p} \partial_x^{\beta} f \|_{L^2(\rr^d)} & \leq C A^{p+\val \beta} (p+1)^{(\delta p+1)\nu} (\val \beta!)^{\mu} \leq C A^{p+\val \beta} (2p)^{(\delta p+1)\nu} (\val \beta!)^{\mu} \\ \notag
& \leq C (2^{\nu}A)^{p+\val \beta} p^{\nu} (2p)^{\delta \nu p} (\val \beta!)^{\mu} \leq C (8^{\nu}e^{\nu}A)^{p+\val \beta} (p!)^{\delta \nu} (\val \beta!)^{\mu},
\end{align}
since for all positive integer $p \geq 1$,
\begin{equation*}
p+1 \leq 2p \leq 2^p \quad \text{and} \quad p^p \leq e^p p!.
\end{equation*}
Notice that from \eqref{int}, since $8^{\nu} e^{\nu} \geq 1$, estimates \eqref{puiss_frac} also hold for $p=0$.  This ends the proof of Lemma~\ref{interpolation}.
\end{proof}
\subsection{Null-controllability and observability of linear evolution equations}\label{control_appendix}
In this section, we recall the link between observability and null-controllability for linear evolution equations. Let us consider a localized control system
\begin{equation}\label{syst_general_annexe}
\left\lbrace \begin{array}{ll}
(\partial_t + P)f(t,x)=u(t,x)\un_{\omega}(x), \quad &  x \in \mathbb{R}^d,\ t>0, \\
f|_{t=0}=f_0 \in L^2(\rr^d),                                       &  
\end{array}\right.
\end{equation}
associated to $P$ be a closed operator on $L^2(\rr^d)$, which is the infinitesimal generator of a strongly continuous semigroup $(e^{-tP})_{t \geq 0}$ on $L^2(\rr^d)$ and where $\omega$ is a measurable subset of $\mathbb{R}^d$. 

By the Hilbert Uniqueness Method, see \cite{coron_book} (Theorem~2.44) or \cite{JLL_book}, the null-controllability of the evolution equation \eqref{syst_general_annexe} is equivalent to the observability of the adjoint system 
\begin{equation} \label{adj_general}
\left\lbrace \begin{array}{ll}
(\partial_t + P^*)g(t,x)=0, \quad & x \in \mathbb{R}^d, \ t>0, \\
g|_{t=0}=g_0 \in L^2(\rr^d),
\end{array}\right.
\end{equation}
where $P^*$ denotes the $L^2(\rr^d)$-adjoint of $P$. 
The notion of observability is defined as follows:

\medskip

\begin{definition} [Observability] Let $T>0$ and $\omega$ be a measurable subset of $\mathbb{R}^d$. 
The evolution equation \eqref{adj_general} is said to be {\em observable from the set $\omega$ in time} $T>0$, if there exists a positive constant $C_T>0$ such that,
for any initial datum $g_0 \in L^{2}(\mathbb{R}^d)$, the mild \emph{(}or semigroup\emph{)} solution of \eqref{adj_general} satisfies
\begin{equation*}\label{eq:observability}
\int\limits_{\mathbb{R}^d} |g(T,x)|^{2} dx  \leq C_T \int\limits_{0}^{T} \Big(\int\limits_{\omega} |g(t,x)|^{2} dx\Big) dt\,.
\end{equation*}
\end{definition}

\medskip
The following result is an adapted Lebeau-Robbiano method which is a simplified formulation of Theorem~3.2 in~\cite{egidi} limited to the semigroup case with fixed control supports and weaker dissipation estimates than in~\cite{KK1} (Theorem~2.1):

\medskip

\begin{theorem}[Beauchard, Egidi \& Pravda-Starov] \label{Meta_thm_AdaptedLRmethod}
 Let $\Omega$ be an open subset of $\mathbb{R}^d$,
 $\omega$ be a measurable subset of $\Omega$,
 $(\pi_k)_{k \geq 1}$ be a family of orthogonal projections on $L^2(\Omega)$,
 $(e^{-tA})_{t \geq 0}$ be a strongly continuous contraction semigroup on $L^2(\Omega)$; 
 $c_1, c_2, c_1', c_2',a, b, t_0, m_1>0 $ be positive constants with $a<b$; $m_2 \geq 0$.
If the following spectral inequality
\begin{equation*} \label{Meta_thm_IS}
\forall g \in L^2(\Omega), \forall k \geq 1, \quad \|\pi_k g \|_{L^2(\Omega)} \leq c_1' e^{c_1 k^a} \|\pi_k g \|_{L^2(\omega)},
\end{equation*}
and the following dissipation estimate with controlled blow-up 
\begin{equation*} \label{Meta_thm_dissip}
\forall g \in L^2(\Omega), \forall k \geq 1, \forall 0<t<t_0, \quad \| (1-\pi_k)(e^{-tA} g)\|_{L^2(\Omega)} \leq \frac{e^{-c_2 t^{m_1} k^b}}{c_2' t^{m_2}} \|g\|_{L^2(\Omega)},
\end{equation*}
hold, then there exists a positive constant $C>1$ such that the following observability estimate holds
\begin{equation*} \label{meta_thm_IO}
\forall T>0, \forall g \in L^2(\Omega), \quad \| e^{-TA} g \|_{L^2(\Omega)}^2 \leq C\exp\Big(\frac{C}{T^{\frac{am_1}{b-a}}}\Big) \int_0^T \|e^{-tA} g \|_{L^2(\omega)}^2 dt.
\end{equation*}
\end{theorem}
\noindent Theorem~\ref{Meta_thm_AdaptedLRmethod} is used in the proof of Theorem~\ref{control_shubin_hd}.

\end{document}